\documentclass[12pt]{amsart}
\usepackage[babel]{csquotes}
\usepackage[shortlabels]{enumitem}
\usepackage{amsmath,amsthm,amssymb,mathrsfs,amsfonts,verbatim,enumitem,color,leftidx,mathtools}
\usepackage{mathabx}
\usepackage{etoolbox} 
\usepackage{bbm}
\usepackage[all,tips]{xy}
\usepackage{graphicx,ifpdf}
\usepackage{stmaryrd}
\usepackage{amssymb}
\usepackage{dsfont}

\usepackage{geometry}
\usepackage{color}
\definecolor{red}{rgb}{1,0,0}

\definecolor{gre}{rgb}{0,0.7,0}

\definecolor{blu}{rgb}{0,0,1}

\ifpdf
\fi
\usepackage[colorlinks]{hyperref}
\hypersetup{
linkcolor=blue,          
citecolor=green,        
}

\usepackage{tikz}

\newtheorem{thm}{Theorem}[section]
\newtheorem{lem}[thm]{Lemma}

\newtheorem{prop}[thm]{Proposition}

\theoremstyle{definition}

\theoremstyle{remark}

\numberwithin{equation}{section}

\definecolor{esperance}{rgb}{0.0,0.5,0.0}

\newcommand{\dd}{\mathrm{d}}

\newcommand{\R}{\mathbb{R}}

\newcommand{\Z}{\mathbb{Z}}
\newcommand{\Q}{\mathbb{Q}}
\newcommand{\C}{\mathbb{C}}

\newcommand{\e}{\mathrm{e}}

\newcommand{\del}{\delta}

\newcommand{\eps}{\epsilon}

\newcommand{\cC}{\mathcal{C}}
\newcommand{\cD}{\mathcal{D}}
\newcommand{\cE}{\mathcal{E}}

\newcommand{\cN}{\mathcal{N}}

\newcommand{\cR}{\mathcal{R}}
\newcommand{\cS}{\mathcal{S}}

\newcommand{\SL}{\operatorname{SL}}
\newcommand{\SO}{\operatorname{SO}}

\newcommand{\Sp}{\operatorname{Sp}}

\newcommand{\ord}{\operatorname{ord}}

\newcommand*{\transp}[2][-1mu]{\ensuremath{\mskip1mu\prescript{\smash{\mathrm t\mkern#1}}{}{\mathstrut#2}}}

\newcommand\set[1]{\left\{#1\right\}}

\newcommand{\onto}{\xymatrix{\ar@{>>}[r]&}}

\newcommand{\eqlabel}[2]
{
\begin{equation}
{#2}\label{#1}
\end{equation}
}

\makeatletter
\newcommand*{\rom}[1]{\expandafter\@slowromancap\romannumeral #1@}
\makeatother

\begin{document}

\title[Values of ternary quadratic forms at integers]{Values of ternary quadratic forms at integers and the Berry-Tabor conjecture for 3-tori}
\author{Wooyeon Kim \and Jens Marklof \and Matthew Welsh}

\address{Wooyeon Kim, June E Huh Center for Mathematical Challenges, Korea Institute for Advanced Study, Seoul, Republic of Korea \newline \rule[0ex]{0ex}{0ex} \hspace{8pt}{\tt wooyeonkim@kias.re.kr}}
\address{Jens Marklof, School of Mathematics, University of Bristol, Bristol BS8 1UG, U.K.\newline \rule[0ex]{0ex}{0ex} \hspace{8pt}{\tt j.marklof@bristol.ac.uk}}
\address{Matthew Welsh, Department of Mathematics, Michigan State University, 619 Red Cedar Road, C212 Wells Hall, East Lansing, MI 48864 U.S.A. \newline \rule[0ex]{0ex}{0ex} \hspace{8pt}{\tt welshmat@msu.edu}}

\date{6 January 2026; revised 5 October 2026}
\thanks{JM's research was supported by EPSRC grant EP/W007010/1. WK was supported by Korea Institute for Advanced Studies (HP101301).}

\begin{abstract}
Berry and Tabor conjectured in 1977 that spectra of generic integrable quantum systems have the same local statistics as a Poisson point process. We verify their conjecture in the case of the two-point spectral density for a quantum particle in a three-dimensional box, subject to a Diophantine condition on the domain's proportions. A permissible choice of width, height and depth is, for example, $1,2^{1/3},2^{-1/3}$. 
This extends previous work of Eskin, Margulis and Mozes (Annals of Math., 2005) in dimension two, where the problem reduces to the quantitative Oppenheim conjecture for quadratic forms of signature $(2,2)$. 
The difficulty in three and higher dimensions is that we need to consider the distribution of indefinite forms in shrinking rather than fixed intervals, which we are able to resolve for special diagonal forms of signature $(3,3)$ in various scalings, including a rate of convergence. A key step of our approach is to represent the relevant counting problem as an average of a theta function on $\SL(2,\Z)^3\backslash\SL(2,\R)^3$ over an expanding family of one-parameter unipotent orbits. The asymptotic behaviour of these unipotent averages follows from Ratner's measure classification theorem and subtle escape of mass estimates. 
\end{abstract}

\maketitle


\section{Introduction}
\label{sec:intro}

\subsection{The Berry-Tabor conjecture}

A principal motivation for this work is the Berry-Tabor conjecture \cite{BerryTabor1977}, which asserts that the statistical correlations in the quantum spectrum of a generic, classically integrable system are that of a Poisson point process, i.e., are maximally uncorrelated. In contrast, the statistics of quantum systems with ``chaotic'' classical limit are expected -- according to Bohigas, Giannoni and Schmit \cite{BGS1984} -- to be governed by one of the Gaussian random matrix ensembles, where the choice of ensemble depends on the time-reversal symmetries of the system. Despite overwhelming numerical evidence and major theoretical advances, there is currently no complete proof of either phenomenon for a single example. 

We will here investigate the validity of the Berry-Tabor conjecture for a specific model, a particle confined to a $d$-dimensional box with edge lengths $\ell_1,\ldots,\ell_d$. The particle's classical motion is along straight lines with elastic reflection at the boundary. This motion lifts to the Kronecker flow on a $d$-torus obtained by reflection on all faces of the box, and is therefore completely integrable. The standard quantisation of this dynamical system is given by the Hamiltonian $H=-\Delta_D$, where $\Delta_D$ is the Dirichlet Laplacian on the cuboid $[0,\ell_1]\times\cdots\times[0,\ell_d]\subset\R^d$. 
The spectrum 
$0<\lambda_1\leq\lambda_2\leq \cdots \to\infty$ of the quantum Hamiltonian $H=-\Delta_D$ is given by the values of the positive definite quadratic form
\begin{equation}\label{Q0}
Q_0(m) = \sum_{j=1}^d  x_j m_j^2 , \qquad x_j = \frac{\pi^2}{\ell_j^2}, 
\end{equation}
where $m=(m_1,\ldots,m_d)$ runs over the integer vectors in $\Z^d$ with positive coefficients. 

The spectrum of the Laplacian for the above $d$-torus (comprising $2^d$ copies of the box) is represented by the same quadratic form, but $m$ now runs over all of $\Z^d$. Our results therefore apply directly to this case. In fact, the reflection symmetry of the rectangular torus under the dihedral group $D_{2h}$ allows a decomposition of its spectrum according to the one-dimensional unitary representations of $D_{2h}$. These are in turn realised by the Laplacian with Dirichlet or Neumann boundary conditions on the various faces of the original box.  

Weyl's law for the spectrum of $H=-\Delta_D$ tells us that
\begin{equation}\label{WeylsLaw}
\#\{ j : \lambda_j < T \} \sim c_d \, T^{d/2} \qquad (T\to\infty),
\end{equation}
where $c_d=(2\pi)^{-d} \ell_1\cdots\ell_d\,\omega_d$ and $\omega_d$ denotes the volume of the unit ball in $\R^d$. In dimension $d=2$ the law is linear in $T$ and hence the average spacing between consecutive elements is asymptotically constant. This is no longer the case in dimension three and higher. Thus, in order to measure correlations on the correct local scale, we unfold the spectrum by setting
\begin{equation}\label{unfold}
\xi_j = c_d \, \lambda_j^{d/2} .
\end{equation}
This simple rescaling produces a sequence whose average spacing is now asymptotic to constant one.

For $T>0$, the two-point (or pair correlation) measure $\cR_T$ of the unfolded, truncated sequence $(\xi_j \mid \xi_j<T)_j$ is defined by
\begin{equation}
\cR_T(A) = \#\{ (i,j) \mid i\neq j, \; \xi_i<T,\; \xi_j \in A+\xi_i\} ,
\end{equation}
with $A\subset\R$ a bounded Borel set. We are interested in the asymptotic behaviour for large $T$, and say that the two-point measure of the full sequence $(\xi_j)_j$ is Poissonian if
\begin{equation}
\lim_{\xi\to\infty} \frac{ \cR_T(A) }{T}  = |A| 
\end{equation}
for any Borel set $A\subset\R$ with $|\partial A|=0$. Here $|A|$ denotes the Lebesgue measure of $A$.
The reason for this terminology is that the above holds almost surely if $(\xi_j)_j$ is a realisation of a Poisson point process on $\R_{>0}$ of intensity one. 

We say that a real number $x$ is $\kappa$-Diophantine for some $\kappa \geq 1$ if there exists a finite constant $C>0$ such that $|x - \tfrac{r}{q}| \geq C q^{-\kappa -1}$ for all integers $r\in\Z$, $q\in\Z_{\geq 1}$. For example, all quadratic surds are $1$-Diophantine, and all irrational algebraic numbers are $\kappa$-Diophantine for every $\kappa>1$ (Roth's theorem). Furthermore, for any given $\kappa>1$, the set of $\kappa$-Diophantine numbers has full Lebesgue measure. 

The following theorem proves the Berry-Tabor conjecture for the second-order statistics of the spectrum of the Hamiltonian $H=-\Delta_D$. Equivalently, it states that the two-point correlations of the values of the degree $d$ homogeneous function $c_d Q_0(m)^{d/2}$ at integers $m\in\Z_{\geq 1}^d$ are Poissonian. 

\begin{thm}
  \label{thm:main}
  Let $(\xi_j)_j$ be the deterministic sequence given by the unfolded spectrum \eqref{unfold} of the Dirichlet Laplacian for a box in $\R^3$ with edge lengths $\ell_1,\ell_2,\ell_3$. If the ratios $\ell_i^2/\ell_j^2$ are $\kappa$-Diophantine for all $i \neq j$ with $\kappa$ sufficiently close to 1, then the pair correlation measure of $(\xi_j)_j$ is Poissonian.
 \end{thm}

 By Roth's theorem, the conditions of the theorem are satisfied, e.g., for $\ell_1= 1$ and algebraic numbers $\ell_2,\ell_3$ so that $\ell_2^2,\ell_3^2,\ell_2^2/\ell_3^2\notin\Q$. This includes the example given in the abstract. The analogous statement was proved by Eskin, Margulis and Mozes \cite{EskinMargulisMozes2005} for positive definite binary quadratic forms $Q_0$, under similar Diophantine conditions, where (as we will discuss further below) this is a special -- and particularly difficult -- case of the quantitative Oppenheim conjecture.
 The necessary unfolding for $d \geq 3$ requires a different approach from that taken in \cite{EskinMargulisMozes2005} or in the more recent effective version by Lindenstrauss, Mohammadi and Wang \cite{LMW2023b}.
 We will here use a similar strategy as in \cite{Marklof2002,Marklof2003} and translate the problem to convergence of unipotent averages over Siegel theta functions associated to $Q_0$. 
Theta functions were also used in connection with quantitative versions of the Oppenheim conjecture, as well as the Davenport-Lewis conjecture, for quadratic forms with five or more variables in the work of G\"otze \cite{Goetze2004} and Buterus, G\"otze, Hille and Margulis \cite{BGHM2022}. 

It is well known that the Diophantine conditions in Theorem \ref{thm:main} are satisfied for $\ell_1=1$ and almost every (in the sense of Lebesgue measure) choice of $\ell_2,\ell_3$. Our theorem therefore strengthens, in dimension $d=3$, recent results by Blomer and Li \cite{BlomerLi2023} proved for almost all rectangular tori.
Previous proofs of the Poisson nature of almost all flat tori are due to \cite{Sarnak1997} in dimension 2, and VanderKam \cite{VanderKam1999a,VanderKam1999b,VanderKam2000} in dimension $d\geq 4$. VanderKam \cite{VanderKam2000} also proves convergence to the Poissonian $k$-point correlation measure, provided $d\geq 2k$.

Triple and higher correlation densities are much less understood in the deterministic setting, i.e., under explicit conditions on the edge lengths $\ell_i$. The best current results for binary forms are due to Aistleitner, Blomer and Radziwi\l\l\ \cite{ABR2024}. The problem of understanding the distribution of gaps $\xi_{j+1}-\xi_j$ appears to be equally difficult; cf.~Blomer, Bourgain, Radziwi\l\l\ and Rudnick \cite{BBRR2017}.

The Diophantine conditions in Theorem \ref{thm:main} are likely not sharp. Note, however, the statement evidently fails for all rational length ratios, when the two-point measure diverges. Following Sarnak \cite{Sarnak1997}, see also \cite{VanderKam1999a,VanderKam1999b,VanderKam2000,Marklof2002,Marklof2002b,Marklof2003}, one can in fact show that the divergence persists along subsequences for a set of second Baire category of edge lengths, i.e., topologically generic boxes.

\subsection{Uniform variant of the quantitative Oppenheim conjecture}

Instead of unfolding the sequence as in \eqref{unfold}, we can also count values in shrinking intervals to compensate for the increasing mean density. This will lead to a uniform variant of the quantitative Oppenheim conjecture of Eskin, Margulis and Mozes \cite{EskinMargulisMozes1998,EskinMargulisMozes2005}, which we will prove here for the special quadratic forms of signature $(3,3)$ given by
\begin{equation}
Q(m) = Q_0(m_1)-Q_0(m_2), \qquad m=(m_1,m_2),
\end{equation}
where $Q_0$ is the positive definite ternary diagonal form in \eqref{Q0}.

Denote by $\cS(\R^n)$ the Schwartz class of functions $\R^n\to\C$. For $M,L>0$, $F\in\cS(\R^{2d})$, $\psi\in\cS(\R)$ we define the counting measure $\cN_{M,L}$ by
\begin{equation}
 \label{eq:NLMdef}
  \cN_{M,L}(F,\psi) 
  =  \sum_{m\in \Z^{2d}} F(M^{-1} m) \psi ( L Q(m)) ,
\end{equation}
and extend this to Borel measurable functions as usual. Thus, if $F$ and $\psi$ are characteristic functions of a unit ball and interval $[a,b]$, respectively, then $\cN_{M,L}$ describes the distribution of values of $Q(m)$ in the interval $[a/L,b/L]$ as $m$ runs over all lattice points $\Z^{2d}$ in the ball of radius $M$. In the classical setting of the quantitative Oppenheim conjecture \cite{EskinMargulisMozes1998,EskinMargulisMozes2005}, $L$ is fixed. A key observation is that in the signature $(2,2)$ case Diophantine conditions on the coefficients are necessary, whereas for signature $(3,1)$ and indefinite forms in more than four variables, only irrationality of the form is required.
However, this changes
if we count in shrinking intervals, i.e., $L\to\infty$ as $M\to\infty$ at suitable rates.

In view of \eqref{WeylsLaw}, the mean spacing between the values of $Q_0(m)$ is about $M^{-(d-2)}$, and hence we expect $\cN_{M,L}(F,\psi)$ to be of size $\asymp M^d\times M^{d-2}/L=M^{2d-2}/L$. The scaling $L=M^{d-2}$ is the critical scaling which captures the local correlations relevant to Theorem \ref{thm:main}, and we will see below that the following statement in fact implies the former as a corollary. 

Let $\e(z)=\exp(2\pi\mathrm{i} z)$, and denote by $D_{2h}$ the symmetry group of a generic rectangular cuboid in $\R^d$, centred at the origin, with $\ord D_{2h}=2^d$. 

\begin{thm}
  \label{thm:main2}
  Fix $d=3$.
  If the ratios $x_i/x_j$ are $\kappa$-Diophantine for all $i \neq j$ with $\kappa$ sufficiently close to 1, then for $\psi\in\cS(\R)$, $F\in\cS(\R^{2d})$, $L,M\geq 1$,
\begin{multline}\label{eq:main2}
\cN_{M,L}(F,\psi)   
= \frac{M^{2d-2}}{L}  \int_{\R} \psi(x) \dd x \;\int_{\R} \bigg( \int_{\R^{2d}} F(v) \e(\xi Q(v)) \dd v \bigg) \dd \xi \\
+ M^d \, \psi(0) \sum_{\sigma\in D_{2h}} \int_{\R^d} F((u,\sigma(u))) \dd u  + O_{F,\psi}\bigg(\frac{M^{2d-2-\tau}}{L}\bigg)+o_{F,\psi}(M^d).
\end{multline}
for any $\tau < \tfrac{1}{\kappa + 1}$. 
\end{thm}

We remark that for the special choice $F(u_1,u_2)=\varphi_1(Q_0(u_1)) \varphi_2(Q_0(u_2))$ with $\varphi_i\in\cS(\R_{\geq 0})$, \eqref{eq:main2} becomes
\begin{multline}\label{eq:main222}
\cN_{M,L}(F,\psi)   
= \frac{M^{2d-2}}{L} \, \frac{d^2}{4}\, \frac{\omega_d^2}{x_1\cdots x_d} \int_{\R} \psi(x) \dd x \; \int_0^\infty \varphi_1(\lambda) \varphi_2(\lambda) \lambda^{d-2} d\lambda  \\
+ 2^dM^d \,\frac{d}{2}\, \frac{\omega_d}{(x_1\cdots x_d)^{1/2}}  \psi(0) \int_0^\infty \varphi_1(\lambda)\varphi_2(\lambda) \lambda^{\frac{d}{2}-1} \dd \lambda  + O_{F,\psi}\bigg(\frac{M^{2d-2-\tau}}{L}\bigg)+o_{F,\psi}(M^d).
\end{multline}
If $L=M^{d-2}$, Theorem \ref{thm:main} now follows from \eqref{eq:main222} by the approximation argument of \cite[Section 2]{Marklof2002}. (Note that we need to take into account the dihedral symmetry, which explains the additional factor of $2^d$.)

As remarked earlier, Eskin, Margulis and Mozes established \eqref{eq:main2} (without a rate) for $L$ fixed and $d\geq 3$ in \cite{EskinMargulisMozes1998} only assuming at least one of the ratios $x_i/x_j$ is irrational, and for $d=2$ \cite{EskinMargulisMozes2005} allowing $\kappa$-Diophantine ratios for any $\kappa\geq 1$.
Also for $d = 2$ and $\kappa$-Diophantine ratios, Lindenstrauss, Mohammadi and Wang have proved a version of (\ref{eq:main2}) with an effective bound for the remainder, allowing for $L$ to grow as a small power of $M$.

The approach taken in \cite{EskinMargulisMozes2005} requires $L$ fixed and the bound for remainder in \cite{LMW2023b} depends quite sensitively on $L$, so taking $L$ as large as $M^{d-2}$ for $d \geq 3$ requires a different strategy.
Here we use Fourier inversion on $\psi$ to transform $\cN_{M,L}(F,\psi)$ into a unipotent average of a theta function.
The advantage is that $L$ appears only in the length of the unipotent average and not implicitly in the test functions.
This feature is used in a similar way in \cite{Marklof2002,Marklof2003} as well as by Buterus, G\"otze, Hille and Margulis \cite{BGHM2022}. 

Buterus, G\"otze, Hille and Margulis \cite{BGHM2022} obtain convergence rates of the form \eqref{eq:main2} for general non-singular quadratic forms (both definite and indefinite) with more than five variables for expanding ($L\to 0$) and shrinking intervals ($L\to\infty$), as $M\to\infty$ under Diophantine conditions similar to those in \cite{EskinMargulisMozes2005}.
While our work provides a stronger rate of convergence in the special case we consider,
the key breakthrough is that we can deal with the scaling regime $M,L\to\infty$ with $L\geq M^{d-2-\tau}$, which is when the second term on the right hand side of \eqref{eq:main2} dominates the error term. This includes the critical scaling $L= M^{d-2}$ required for the proof of Theorem \ref{thm:main}.
A further advantage of our approach, particularly relevant to the calculation of the second term, is the deployment of theta sums with general cut-off functions $F$, via the Shale-Weil representation, not only classical theta series with Gaussian weights as in \cite{BGHM2022}.

As mentioned above, a rate of convergence for signature $(2,2)$, in the critical scaling ($L$ fixed), was recently obtained by Lindenstrauss, Mohammadi and Wang \cite{LMW2023b} as an application of their ground-breaking work on effective Ratner equidistribution \cite{LMW2022,LMW2023a}. 
Str\"ombergsson and Vishe \cite{AV2020} obtained rates of convergence for the quantitative Oppenheim conjecture for special inhomogeneous forms studied in \cite{Marklof2002,Marklof2003}. The quantitative Oppenheim conjecture for general inhomogeneous forms was proved by Mohammadi and Margulis \cite{MargulisMohammadi2011} without rate of convergence. More is known if one restricts to error rates valid for almost every quadratic form of a given signature; see the papers by Bourgain \cite{Bourgain2016}, Ghosh and Kelmer \cite{GhoshKelmer2017,GhoshKelmer2018}, Athreya and Margulis \cite{AthMar2018}, Kelmer and Yu \cite{KelmerYu2020}, and Ghosh, Kelmer and Yu \cite{GKY2022,GKY2023}.

\subsection{Escape of mass and Margulis functions}

Our strategy for the proof of Theorem \ref{thm:main2} is to relate $\cN_{M,L}(F,\psi)$ to homogeneous dynamics on $\Gamma^d \backslash G^d$, $\Gamma = \SL(2, \Z)$, $G = \SL(2, \R)$, via a theta function $\widetilde\Theta_F$ defined on the semi-direct product of a Heisenberg group with $G^d$ using the Shale-Weil representation. The connection between the counting function $\cN_{M,L}(F,\psi)$ and theta function $\widetilde\Theta_F$ is given by
\begin{equation}
  \label{eq:FourierRLM}
  \begin{split}
\cN_{M,L}(F,\psi) 
& = \frac{1}{L} \int_{\mathbb{R}} \hat{\psi}( L^{-1}\xi)  \bigg(\sum_{m \in \mathbb{Z}^{2d}} F(M^{-1} m) \e\big(\xi Q(m) \big) \bigg)  \dd \xi \\
&  =  \frac{M^{d}}{L} \int_{\mathbb{R}} \hat{\psi} ( L^{-1}\xi ) \widetilde\Theta_F\big(1, g_0 n(2\xi) a(\log M) g_0^{-1}\big)  \dd \xi,
\end{split}
\end{equation}
where $\hat{\psi}$ is the Fourier transform of $\psi$,
\begin{equation}
  \label{eq:psihatdef}
  \hat{\psi}(\xi) = \int_{\mathbb{R}} \psi(x) \e(- \xi x) \dd x .
\end{equation}
The theta function is evaluated at the product of the elements
\begin{equation}
  \label{eq:g0andef}
  g_0 = \left(
  \begin{pmatrix}
    x_j^{\frac{1}{2}} & 0 \\
    0 & x_j^{-\frac{1}{2}}
  \end{pmatrix} 
    \right)_{\!\! j\leq d} \in G^d,
\end{equation}
and
\begin{equation}
  \label{eq:nadef}
  n(\xi) =
  \begin{pmatrix}
    1 & \xi \\
    0 & 1
  \end{pmatrix}
  ,
  \qquad 
   a(t) =
  \begin{pmatrix}
    \e^{-t}& 0 \\
    0 & \e^{t}
  \end{pmatrix}
  \in G,
\end{equation}
where we identify $G$ with its diagonal embedding in $G^d$.

Our analysis of the unipotent average over $\xi$ in \eqref{eq:FourierRLM} is significantly complicated by the unboundedness of the theta function.
In fact, the first term on the right of \eqref{eq:main2} comes from the contribution of $\xi$ of size at most $\tfrac{\epsilon}{M}$ for a small $\epsilon > 0$, see Section \ref{sec:smallxi}.
We view this part of the integral as analogous to a ``major arc'' in the circle method. This is analogous to the approach taken in \cite{Marklof2002,Marklof2003} and \cite{BGHM2022}.

The remaining $\xi$, the ``minor arc,'' contributes the second term in \eqref{eq:main2}, which represents the ``diagonal'' term in the two-point density.
This is obtained from the equidistribution of the translated unipotent segment.
As results on the equidistribution of unipotent flows, e.g. Ratner's theorems, typically require bounded functions, we need to control the escape of mass in \eqref{eq:FourierRLM}.
We prove the following new upper bound for the Margulis function on $\Gamma^3 \backslash G^3$, which is of independent interest.
\begin{thm}
  \label{thm:escapeofmass}
  Fix $d = 3$.
  If the ratios $x_i/x_j$ are $\kappa$-Diophantine for all $i \neq j$ with $\kappa$ sufficiently close to 1,
 then for any $0 < \tau < \tfrac{1}{\kappa + 1}$, $0 \leq \theta < \frac{3}{2\kappa} - 1$, $\eta >0$ sufficiently small and $1 \leq L \leq M^{(d-2) + \theta} $, there is $\beta > 1$ such that
  \begin{equation}
    \label{eq:escapeofmass}
    \frac{1}{L} \int_{M^{-1-\eta} < |\xi| \leq L} \alpha_3(g_0n(\xi)a(\log M))^\beta \dd \xi =
    O_{\kappa,\eta, \theta}\bigg(1+\frac{1}{L} M^{d-2 - \tau}\bigg).
  \end{equation}
\end{thm}

Here $\alpha_d(g)$ is a Margulis function, introduced in the proof of the quantitative Oppenheim conjecture \cite{EskinMargulisMozes1998}. It measures how high $\Gamma^d g$ is in the cusp of $\Gamma^d \backslash G^d$. As for Theorem \ref{thm:main2}, we expect the above bound to be true also for any  $d\geq 4$ under suitable Diophantine conditions.

  We note that for $L > M^{1 - \tau}$ the average (\ref{eq:escapeofmass}) is bounded.
  In comparison, \cite{BGHM2022} obtains bounds for similar (more general) averages of size roughly $M^{3\beta -2}$, which holds without Diophantine conditions.
  The stronger estimate (\ref{eq:escapeofmass}) in the special case we consider is the main source for the improvements in Theorem \ref{thm:main2} over the results of \cite{BGHM2022}.
  We remark that Diophantine conditions are used in \cite{BGHM2022} to obtain a pointwise bound for the Margulis functions, whereas the proof of Theorem \ref{thm:escapeofmass} uses the Diophantine condition in an essential way in controlling the average.

The proof of Theorem \ref{thm:escapeofmass} is given in Section \ref{sec:escape}.
Our strategy begins similarly to \cite[Section~5]{EskinMargulisMozes1998}.
First a global contraction of $\alpha_3(g)^\beta$ under translates of unipotent segments (Proposition \ref{prop:globalcontraction}) is proved by combining purely local calculations with geometric properties of the cusp of $\Gamma^3 \backslash G^3$, and then \eqref{eq:escapeofmass} is proved by iterating the global contraction.

However, a straightforward adaptation of \cite{EskinMargulisMozes1998} only establishes \eqref{eq:escapeofmass} for $\beta < \tfrac{2}{3}$ while $\beta > 1$ is required for \eqref{eq:FourierRLM}.
The lack of contraction for $\beta > \tfrac{2}{3}$ is already evident in the local calculations and is related to possible repeated roots of a certain polynomial.
We regain contraction away from an exceptional set by modifying the Margulis function to account for the distance between roots of this polynomial using a strategy developed in \cite{kim2025momentsmargulisfunctionsindefinite}.
The bound \eqref{eq:escapeofmass} for $\beta > 1$ is then obtained by iterating the contraction and by controlling the contribution of the exceptional set by an avoidance estimate (Proposition \ref{prop:avoidanceMargulis}).
Our proof of this avoidance estimate uses the above-mentioned work by Lindenstrauss, Mohammadi, and Wang \cite[Theorem~1]{LMW2023a} on the polynomial effective equidistribution of one-parameter unipotent flows on $\Gamma^2 \backslash G^2$.

\section{Preliminaries}
\label{sec:prelims}

\subsection{Theta functions and Shale-Weil representation}
\label{sec:oscillatortheta}

The Heisenberg group $H_d$ is defined to be the set $\R^d \times \R^d \times \R$ with multiplication given by
\begin{equation}
  \label{eq:Hmult}
  (x_1, y_1, z_1)(x_2, y_2, z_2) 
  = (x_1 + x_2, y_1 + y_2, z_1 + z_2 - \tfrac{1}{2}(x_1\transp{y_2} - x_2\transp{y}_1)),
\end{equation}
where vectors in $\R^d$ are represented as rows.
Every $g\in\Sp(d,\R)$ acts on $H_d$ by automorphisms, since it
preserves the symplectic form
$x_1\transp y_2-x_2\transp y_1$. Via the embedding
\eqref{eq:gembed}, we regard $G^d$ as a subgroup of
$\Sp(d,\R)$ and restrict this action to $G^d$:
\begin{equation}
  \label{eq:gautomorphism}
  h^g = \left(
  \begin{pmatrix}
    x & y
  \end{pmatrix}
  g,
  z\right),\ \mathrm{where\ } h= (x, y, z).
\end{equation}
The group $H_d \rtimes G^d$ is then defined to have multiplication
\begin{equation}
  \label{eq:Jacobimult}
  (h_1, g_1)(h_2, g_2) = (h_1 h_2^{g_1^{-1}}, g_1g_2). 
\end{equation}

The Schr\"odinger representation $W$ of $H_d$ acts on $L^2(\R^d)$ by the unitary transformations
\begin{equation}
  \label{eq:Schrodingerdef}
  [W_d(x, y, z)f](t) = \e(t\transp{y} -z + \tfrac{1}{2} x\transp{y}) f(t + x).
\end{equation}
For $g \in G^d$, we obtain another unitary representation $W_d^g$ of $H_d$ by $W_d^g(h) = W_d(h^g)$. By the Stone-von Neumann theorem on the unitary equivalence of infinite-dimensional, irreducible representations of $H_d$, there is an operator $R_d(g)$ that intertwines $W_d$ and $W_d^g$, i.e., for which
\begin{equation}
  \label{eq:Rintertwine}
  W_d^g = R_d(g)^{-1} W_d R_d(g) .
\end{equation}
These $R_d(g)$ yield the projective Shale-Weil representation of $\Sp(d,\R)$, which is central in the construction of Siegel theta functions associated with a quadratic form in $d$ variables. Since we are dealing here with diagonal forms only, it suffices to consider a subgroup of $\Sp(d, \R)$, which we identify with the $d$-power of $G=\SL(2,\R)=\Sp(1,\R)$ via the embedding
\begin{equation}
  \label{eq:gembed}
  G^d \to \Sp(d,\R),\qquad 
  \left(
  \begin{pmatrix}
    a_j & b_j \\
    c_j & d_j 
  \end{pmatrix}
  \right)_{1\leq j \leq d} \mapsto
  \begin{pmatrix}
    a_1 &  & & b_1 & &  \\
        & \ddots & & & \ddots & \\
        & & a_d & & & b_d \\
    c_1 &  & & d_1 & &  \\
        & \ddots & & & \ddots & \\
        & & c_d & & & d_d \\
  \end{pmatrix}
  .
\end{equation}
The Shale-Weil representation $R_d$ restricted to $G^d$ is then given as the $d$th tensor product of $R_1$. That is, for $g = (g_1,\ldots,g_d) \in G^d$, we have
\begin{equation}
R_d(g) = R_1(g_1) \otimes \cdots \otimes R_1(g_d).
\end{equation}
Proposition \ref{prop:oscillator} gives an explicit method for computing unitary operators $R_1(g)$ on a dense subset of $L^2(\R)$.

\begin{prop}
  \label{prop:oscillator}
  Let $f\in\cS(\R)$.
  For 
  \begin{equation}
    \label{eq:poscillator}
    p =
    \begin{pmatrix}
      a & b \\
      0 & a^{-1}
    \end{pmatrix}
  \end{equation}
  we have
  \begin{equation}
    \label{eq:poscillator1}
    R_1(p)f(x) = |a|^{\frac{1}{2}} \e( \tfrac{1}{2}  ab x^2) f(ax).
  \end{equation}
  Moreover, for $ g = 
  \begin{pmatrix}
    a & b \\
    c & d
  \end{pmatrix}
  \in G$ with $c \neq 0$, we have
  \begin{equation}
    \label{eq:Fouriertransform}
    R_1(g)f(x) = |c |^{-\frac{1}{2}} \e ( \tfrac{a}{2c}x^2) \int_{\R} f(y) \e( \tfrac{d}{2c}y^2 - \tfrac{1}{c} x y) \dd y.
  \end{equation}
\end{prop}

For $f\in\cS(\R^d)$, the theta function $\Theta_f$ is defined by
\begin{equation}
  \label{eq:thetadef}
  \Theta_f(h, g) = \sum_{m \in \Z^d} [W_d(h)R_d(g)f](m).
\end{equation}
We observe that for $1 =(0,0,0) \in H_d$ and $g = \left(
\begin{pmatrix}
  1 & x_j \\
  0 & 1
\end{pmatrix}
\begin{pmatrix}
  y_j^{\frac{1}{2}} & 0 \\
  0 & y_j^{-\frac{1}{2}}
\end{pmatrix}
\right)_{\!\!1\leq j\leq d}$,
\begin{equation}
  \label{eq:thetaex}
  \Theta_f(1, g)
  = (y_1 \cdots y_d)^{\frac{1}{4}}\sum_{m \in \Z^d} f\left(( m_j y_j^{\frac{1}{2}})_j\right) \e\bigg( \tfrac{1}{2} \sum_{j=1}^d x_jm_j^2 \bigg).
\end{equation}
We now extend the above to a unitary representation $R_{d,d}(g)$ of $G^d$ on $L^2(\R^{2d})$
as the tensor product
\begin{equation}
R_{d,d}(g) = R_d(g) \otimes [C R_d(g) C], 
\end{equation}
where $C$ denotes complex conjugation.
We similarly extend $W_d$ to $W_{d,d}$.
We highlight that $R_{d,d}$ is (unlike $R_1$, say) an actual representation for all $d$. The corresponding theta function is defined, for $F\in\cS(\R^{2d})$, by
\begin{equation}
  \label{eq:thetadefdd}
  \widetilde\Theta_F(h, g) = \sum_{m \in \Z^{2d}} [W_{d,d}(h)R_{d,d}(g)F](m).
\end{equation}
For example, for $F=f\otimes \overline f$ we have
$\widetilde\Theta_F(h, g) = |\Theta_f(h, g)|^2$.
Furthermore, for $g$ as in \eqref{eq:thetaex}, we have the explicit formula 
\begin{equation}
  \label{eq:thetaexdd}
  \widetilde\Theta_F(1, g)
  = (y_1 \cdots y_d)^{\frac{1}{2}}\sum_{(m,n) \in \Z^{2d}} F\left(( m_j y_j^{\frac{1}{2}})_j,( n_j y_j^{\frac{1}{2}})_j\right) \e\bigg( \tfrac{1}{2} \sum_{j=1}^d x_j(m_j^2-n_j^2) \bigg).
\end{equation}

\subsection{Automorphy of theta functions}
\label{sec:automorphy}

Define
\begin{equation}
  \label{eq:Gammabar}
  \bar{\Gamma} = \{ ((m, n, z)h_\gamma, \gamma) \in H_d\rtimes G^d : m, n \in \Z^d,\ z \in \R, \gamma \in \Gamma^d\},
\end{equation}
where for $\gamma = (
\begin{pmatrix}
  a_j & b_j \\
  c_j & d_j
\end{pmatrix}
)_j,$ $h_\gamma = ((\tfrac{1}{2} c_jd_j)_j, (\tfrac{1}{2}a_jb_j)_j , 0)$.
We note that modulo left multiplication by $\Z^d \times \Z^d \times \R \subset H_d$,
\begin{equation}
  \label{eq:hgammamult}
  h_{\gamma_1 \gamma_2} \equiv h_{\gamma_1} (h_{\gamma_2})^{\gamma_1^{-1}},
\end{equation}
so \eqref{eq:Gammabar} does define a subgroup. The following theorem follows immediately from \cite[Theorem~4.1]{MarklofWelsh2021a}; cf.~also \cite{Marklof2003}.

\begin{thm}
  \label{thm:thetaautomorphy}
  For $F\in\cS(\R^{2d})$,  $\bar\gamma \in \bar{\Gamma}$ and $(h,g) \in H_d \rtimes G^d$,
  \begin{equation}
    \label{eq:thetaautomorphy}
     \widetilde\Theta_F\big(\bar\gamma (h, g)\big) =  \widetilde\Theta_F(h, g) .
  \end{equation}
\end{thm}

Note that in particular 
\begin{equation}
\label{thetainversion}
\widetilde\Theta_F\big((h_\gamma, \gamma)(1, g)\big) =  \widetilde\Theta_F(1, g)
\end{equation}
for all $\gamma\in\Gamma^d$.
This implies that $\widetilde\Theta_F(1, \,\cdot\,) $ defines a function on $\Gamma_{2,2}^d\backslash G^d$ with the finite index subgroup $\Gamma_{2,2}^d<\Gamma^d$ such that $c_jd_j \equiv a_jb_j \equiv 0 \bmod 2$. 

\subsection{Margulis functions}
\label{sec:Margulisfunctions}

We let $\cD$ be the standard fundamental domain for $\Gamma \backslash G$, which we take as $\cD = \cD^+ \cup \cD^-$ with
\begin{equation}
  \label{eq:standardfunddom}
  \begin{split}
    \cD^+ = & \left\{
    \begin{pmatrix}
      1 & x \\
      0 & 1
    \end{pmatrix}
    \begin{pmatrix}
      y^{\frac{1}{2}} & 0 \\
      0 & y^{-\frac{1}{2}}
    \end{pmatrix}
    k(\theta)    : 
    0 \leq x \leq \tfrac{1}{2},\ x^2 + y^2 \geq 1,\ \theta \in [0, \pi)\right\}. \\
    \cD^- = & \left\{
    \begin{pmatrix}
      1 & x \\
      0 & 1
    \end{pmatrix}
    \begin{pmatrix}
      y^{\frac{1}{2}} & 0 \\
      0 & y^{-\frac{1}{2}}
    \end{pmatrix}
    k(\theta)          : 
   -\tfrac{1}{2} < x < 0,\ x^2 + y^2 > 1,\ \theta \in [0, \pi)\right\},
  \end{split} 
\end{equation}
where
\begin{equation}
k(\theta) = \begin{pmatrix}
      \cos \theta & -\sin \theta \\
      \sin \theta & \cos \theta
    \end{pmatrix} .
\end{equation}
Under the Möbius action of $G$ on the upper half-plane,
if
\begin{equation}
  \label{eq:IwasawaCoords}
g=
\begin{pmatrix}
1&x\\
0&1
\end{pmatrix}
\begin{pmatrix}
y^{1/2}&0\\
0&y^{-1/2}
\end{pmatrix}
k(\theta),
\end{equation}
then $g\cdot i=x+ i y$, since $k(\theta)$ fixes
$i$. 

For $g = \left(
\begin{pmatrix}
  a_j & b_j \\
  c_j & d_j
\end{pmatrix}
\right)_{\!\!1\leq j\leq d} \in G^d$, we define the height functions
\begin{equation}
  \label{eq:rholdef}
  \rho_l(g) = \min_{\substack{J \subset \{1, \dots, d\} \\ |J| = l}} \prod_{j \in J} |(c_j, d_j)|,
\end{equation}
and
\begin{equation}
  \label{eq:alphaldef}
  \alpha_l(g) = \max_{\gamma \in \Gamma^d} \rho_l(\gamma g)^{-1}.
\end{equation}
We also use $\rho(g) := \rho_1(g)$, $\alpha(g) := \alpha_1(g)$ when $g \in G$.
We note that in the coordinates \eqref{eq:IwasawaCoords}, \(\rho(g)=y^{-1/2}\), so $\rho(g)^{-1}=y^{1/2}$ is the square root of the usual
height in the cusp.

For $g = (g_j)_j\in G^d$, we let $\gamma(g) = (\gamma_j(g))_j$ be so that $\gamma_j(g) g_j \in \cD$. We drop the argument and simply write $\gamma$ and $\gamma_j$ in the following.
We record the following elementary properties, which follow directly from the choice of fundamental domain $\cD$.
\begin{lem}
  \label{lem:shortestvector}
  Let $g = (g_j)_j$ and $\gamma = (\gamma_j)_j$.
  Then
  \begin{enumerate}[{\rm (i)}]
  \item $\alpha_l(g) = \max_{\substack{J \subset \{1, \dots, d\} \\ |J| = l}} \prod_{j \in J} \rho(\gamma_j g_j)^{-1}$,
    
  \item $\rho(\gamma_j g_j) \ll 1$, and
    
  \item if $\rho(g_j) < 1$, then $\gamma_j$ has the form $
  \pm \begin{pmatrix}
    1 & * \\
    0 & 1
  \end{pmatrix}
  $.
  \end{enumerate}
\end{lem}

\begin{lem}
  \label{lem:basicfundamentaldomain}
  Set $g_0 n(\xi) a(t) = (g_j)_j$ and write $\gamma( g_0 n(\xi) a(t)) = \left(
  \begin{pmatrix}
    * & * \\
    c_j & d_j
  \end{pmatrix}
 \right)_{\!\!1\leq j\leq d}$.
  \begin{enumerate}[{\rm (i)}]    
  \item
    If $\rho(\gamma_jg_j) \leq \delta_j$, then we have
    \begin{equation}
      \label{eq:cjdjbounds}
      |c_j|\ll\delta_j\e^t,
    \qquad
    |c_jx_j\xi+d_j|\ll\delta_j\e^{-t}.
    \end{equation}
 
  \item There exists $t_0 > 0$, depending only on $x_1, x_2,\ldots,x_d$, such that, if $t \geq t_0$, then $c_j \neq 0$.
    
  \item There exists a constant $\epsilon_0 > 0$ such that, if $t \geq t_0$ and $|\xi| \geq \epsilon_0 \e^{-t}$, then $d_j \neq 0$ for all $1\leq j\leq d$. 
  \end{enumerate}
\end{lem}

\begin{proof}
By Lemma~\ref{lem:shortestvector} (ii), there exists a constant $C=C(\cD)>0$ such that $\rho(h)\leq C$ for every
$h\in\cD$.

As the bottom row of $\gamma_jg_j$ is
\[
\left(
c_jx_j^{1/2}\e^{-t},
\,
x_j^{-1/2}\e^t(c_jx_j\xi+d_j)
\right),
\]
the assertion (i) follows immediately.

For (ii), suppose that $c_j=0$. Since the bottom row of
$\gamma_j$ is a primitive vector, $|d_j|=1$, and hence \(\rho(\gamma_jg_j)=x_j^{-1/2}\e^t\). This contradicts $\rho(\gamma_jg_j)\leq C$ for sufficiently large $t$. Taking the maximum of the resulting
thresholds over $1\leq j\leq d$ proves (ii).

Finally, if $d_j=0$, then $|c_j|=1$, and therefore \(\rho(\gamma_jg_j)
\geq x_j^{1/2}|\xi|\e^t\). Since $\rho(\gamma_jg_j)\leq C$, this implies \(|\xi|\leq Cx_j^{-1/2}\e^{-t}\).
Choosing \(\epsilon_0>
C\max\{x_1^{-1/2},x_2^{-1/2},\ldots, x_d^{-1/2}\}\) proves (iii).
\end{proof}

The functions $\alpha_l$ play an important role in the behavior of the theta functions.
The following is a consequence of estimates for the Shale-Weil representation combined with elementary estimates for sums over lattices, see \cite[Theorem~4.4]{MarklofWelsh2021a} and \cite{Marklof2003}.

For $f\in\cS(\R^d)$, we write
  \begin{equation}
    \label{eq:ftheta}
    f_\theta = R_d\left(\left(
    \begin{pmatrix}
      \cos \theta_j & -\sin \theta_j \\
      \sin \theta_j & \cos \theta_j
    \end{pmatrix}
   \right)_{\!\!1\leq j\leq d}\right)f ,
  \end{equation}
and similarly for $F\in\cS(\R^{2d})$,
  \begin{equation}
    \label{eq:fthetadd}
    F_\theta = R_{d,d}\left(\left(
    \begin{pmatrix}
      \cos \theta_j & -\sin \theta_j \\
      \sin \theta_j & \cos \theta_j
    \end{pmatrix}
   \right)_{\!\!1\leq j\leq d}\right) F.
  \end{equation}

\begin{thm}
  \label{thm:thetaasymptotic}
  Let $F\in\cS(\R^{2d})$ and $A > 0$. Then for $h=(0,y,z)\in H_d$ and
  $$g=(g_j)_{1\leq j\leq d}, \quad g_j
  = 
  \begin{pmatrix}
    1 & u_j \\
    0 & 1
  \end{pmatrix}
  \begin{pmatrix}
    v_j^{\frac{1}{2}} & 0 \\
    0 & v_j^{-\frac{1}{2}}
  \end{pmatrix}
  \begin{pmatrix}
    \cos \theta_j & -\sin \theta_j \\
    \sin \theta_j & \cos \theta_j
  \end{pmatrix}
  \in\cD,$$ we have
  \begin{equation}
    \label{eq:thetaasymp}
    \widetilde\Theta_F(h, g) = (v_1 \cdots v_d)^{\frac{1}{2}} F_\theta(0) + O_{F,A}\big((v_1 \cdots v_d)^{\frac{1}{2}} (\min v_j)^{-A}\big) 
  \end{equation}
  and, for all $h\in H_d$, $g\in G^d$,
 \begin{equation}
    \label{eq:thetaasymp123}
    |\widetilde\Theta_F(h, g)|  \ll_F \alpha_d(g).
  \end{equation}
\end{thm}

In what follows we use the following estimate on the change of $\rho$ under translation by $n(\xi)$ and $a(s)$.
\begin{lem}
  \label{lem:rhotransbounds}
  For any $g, h \in G$, we have
  \begin{equation*}
    |h|^{-1} \rho(g) \leq \rho(gh) \leq |h| \rho(g)
  \end{equation*}
  with $|\cdot|$ the operator norm.
  In particular, for $|\xi|\leq 1$ and $s > 0$, we have
  \begin{equation}
    \label{eq:rhotranslationbounds}
    \phi^{-1} \e^{-s}  \rho(g) \leq \rho(g n(\xi)a(s)) \leq \phi \e^s \rho(g)
  \end{equation}
  with $\phi = \tfrac{1 + \sqrt 5}{2}$. 
\end{lem}

\begin{proof}
  For $g =
  \begin{pmatrix}
    * & * \\
    c & d
  \end{pmatrix}
  $, we have
  \begin{equation}
    \label{eq:rhoTranslateBound}
    \rho(gh) = |(c, d)h| \leq |h| |(c,d)| = |h| \rho(g). 
  \end{equation}
  This establishes the upper bounds; the lower bounds follow similarly. 
\end{proof}

\subsection{Diophantine property}
\label{sec:Diophantine}

In this subsection, we specialize to $d=2$ or $d=3$, as specified in each statement.

\begin{lem}
  \label{lem:alpha3supbound}
  Assume that $g_0 = \left( \begin{pmatrix}   x_j^{\frac{1}{2}} & 0 \\
    0 & x_j^{-\frac{1}{2}}
  \end{pmatrix}
  \right)_{\!\! j=1,2,3}$ with $\tfrac{x_i}{x_j}$ $\kappa$-Diophantine for $i\neq j$.
  Then for any $\eta > 0$, if $\e^{-(1+\eta) t} < |\xi| \leq 10$, then $\alpha_3(g_0n(\xi) a(t)) \ll \e^{(\frac{3\kappa}{\kappa + 1} + 3\eta)t }$.
\end{lem}

\begin{proof}
It suffices to consider $t$ sufficiently large. Write
$g_0n(\xi)a(t)=(g_j)_{j=1}^3$ and
\[
\gamma\big(g_0n(\xi)a(t)\big)
=
\left(
\begin{pmatrix}
*&*\\
c_j&d_j
\end{pmatrix}
\right)_{\!\!j=1,2,3},
\qquad
\delta_j=\rho(\gamma_jg_j).
\]
By Lemma~\ref{lem:basicfundamentaldomain} (ii), $c_j\neq0$ for
all $j$. Moreover,
\begin{equation}
\label{eq:cjdjbounds-used}
|c_j|\ll\delta_j\e^t,
\qquad
|c_jx_j\xi+d_j|\ll\delta_j\e^{-t}.
\end{equation}
Since $|\xi|\leq10$, these estimates also give
$|d_j|\ll\delta_j\e^t$.

Suppose that $d_j\neq0$, and let $i\neq j$. Since $c_id_j\neq0$, the Diophantine condition gives
\begin{equation}
\label{eq:deltalowerbound1}
(\delta_i\delta_j\e^{2t})^{-\kappa}
\ll
\left|c_id_j\frac{x_i}{x_j}-c_jd_i\right|.
\end{equation}
On the other hand,
\[
c_id_j\frac{x_i}{x_j}-c_jd_i
=
c_i\frac{x_i}{x_j}(c_j x_j\xi+d_j)-c_j(c_i x_i\xi+d_i),
\]
and hence \eqref{eq:cjdjbounds-used} implies
\[
\left|c_id_j\frac{x_i}{x_j}-c_jd_i\right|
\ll\delta_i\delta_j.
\]
It follows that
\begin{equation}
\label{eq:deltapairbound}
(\delta_i\delta_j)^{-1}
\ll \e^{\frac{2\kappa}{\kappa+1}t}.
\end{equation}

If $d_j=0$, then $|c_j|=1$. The second coordinate of the
bottom row of $\gamma_jg_j$ therefore gives
\begin{equation}
\label{eq:dj0lowerbound}
\delta_j^{-1}
\ll |\xi|^{-1}\e^{-t}
\ll \e^{\eta t}.
\end{equation}
Notice also that, since $c_j\neq0$, the first coordinate gives
the general bound \(\delta_j^{-1}\ll\e^t\).

If all three $d_j$ are nonzero, multiplying
\eqref{eq:deltapairbound} over the three pairs, we have
\[
(\delta_1\delta_2\delta_3)^{-1}
\ll \e^{\frac{3\kappa}{\kappa+1}t}.
\]
If exactly one $d_j$ vanishes, we combine
\eqref{eq:deltapairbound} for the other two indices with
\eqref{eq:dj0lowerbound}. If at least two of the $d_j$ vanish,
we use \eqref{eq:dj0lowerbound} for those indices and the
general bound $\delta_j^{-1}\ll\e^t$ for the remaining one.
Since $\kappa\geq1$, all cases combined give
\[
\alpha_3(g_0n(\xi)a(t))
=(\delta_1\delta_2\delta_3)^{-1}
\ll
\e^{(\frac{3\kappa}{\kappa+1}+3\eta)t},
\]
which completes the proof.
\end{proof}

The following lemma connects the classical Diophantine notion with that used in \cite[Theorem~1]{LMW2023a}. In the following, for any $d\geq 2$, denote by $H$ the group $\operatorname{SL}(2, \mathbb{R})$ embedded diagonally in $G^d$.
\begin{lem}\label{lem:periodicorbits}
    Let $g_0=\left(\left(\begin{matrix}
        x_j^{\frac{1}{2}} & 0 \\ 0 & x_j^{-\frac{1}{2}}
    \end{matrix}\right)\right)_{\!\!j=1,2}\in G^2$ and assume that $\frac{x_1}{x_2}$ is $\kappa$-Diophantine.
Then for any $\xi\in \R$ and for any closed orbit $gH\subset  \Gamma^2\backslash G^2$,
    $$\operatorname{dist}_{\Gamma^2 \backslash G^2}( g_0n(\xi), gH)\gg (\operatorname{Vol}(gH))^{-D}$$
    for some constant $D = D(\kappa) >1$.
\end{lem}
\begin{proof}
  Let $$\tilde n(\xi_1, \xi_2) =
  \left(
  \begin{pmatrix}
    1 & \xi_j \\
    0 & 1
  \end{pmatrix}
\right)_{\!\! j=1,2} . $$
 We may write $\Gamma^2 g_0 n(\xi) = \Gamma^2 g_0 \tilde n(\xi_1, \xi_2)$, for suitable bounded $(\xi_1, \xi_2)\in\R^2$.
  It is therefore sufficient to show that $\operatorname{dist}_{\Gamma^2 \backslash G^2} (g_0 \tilde n(\xi_1, \xi_2), g H) \gg (\operatorname{Vol}(g H))^{-D}$
  for $|\xi_1|, |\xi_2| \ll 1$. 
  
  We may assume $g = (\mathrm{id}, h)$.
  We note that if for $A \in \R^{2 \times 2}$, $g_2 A g_1^{-1} =A$ for all $(g_1,g_2) \in g H g^{-1}$, then $A$ is a multiple of $h$.
  Since $\Gamma^2 \cap g H g^{-1}$ is Zariski dense, then the same holds if $\gamma_2 A \gamma_1^{-1} = A$ for all $(\gamma_1, \gamma_2) \in \Gamma^2 \cap g H g^{-1}$.
  Since $\gamma_2 h \gamma_1^{-1} = h$ for $(\gamma_1, \gamma_2) \in \Gamma^2 \cap g H g^{-1}$ defines a system of integer linear equations with a one-dimensional set of solutions, it follows that $h = (\det A)^{-\frac{1}{2}} A$ with $A = (a_{ij})\in \Z^{2\times 2}$.

  We may assume the entries of $A$ are coprime.  
 Note $\Gamma \cap h^{-1} \Gamma h$ is conjugate in $\Gamma$ to $\Gamma_0(\det A)$, where
  $$\Gamma_0(N) = \left\{
  \begin{pmatrix}
    a & b \\
    c & d
  \end{pmatrix}
  \in \Gamma : c \equiv 0 \bmod N\right\}
  $$ is the Hecke congruence subgroup of level $N$.
  The index of $\Gamma_0(\det A)$ in $\Gamma$ is $\geq |\det(A)|$, so $|\det A| \ll \operatorname{Vol}(gH)$.

  Suppose $\operatorname{dist}_{\Gamma^2 \backslash G^2} (g_0 \tilde n(\xi_1, \xi_2), g H)  \leq \epsilon$.
  Then there is $\gamma \in \Gamma^2$, $h_1 \in G$ with $|h_1 - \mathrm{id}| \ll \epsilon$, and $h_2 \in G$ such that $g_0 \tilde n(\xi_1, \xi_2) (\mathrm{id}, h_1)= \gamma (\mathrm{id}, h) (h_2, h_2)$.
  It follows that
  \begin{equation}
    \label{eq:x1x2diophantineapprox}
    \begin{pmatrix}
      (\tfrac{x_1}{x_2})^{-\frac{1}{2}} & (x_1 x_2)^{\frac{1}{2}} (\xi_2 - \xi_1) \\
      0 & ( \tfrac{x_1}{x_2})^{\frac{1}{2}}
    \end{pmatrix}
    = \frac{1}{\sqrt{\det A}}\, A' + O(\epsilon)
  \end{equation}
  for $A' \in \Z^{2\times 2}$.
  Inspecting the $(2,2)$-entry, the Diophantine condition on $\tfrac{x_1}{x_2}$ implies $\epsilon \gg (\det A)^{-\kappa -1}$ and the lemma follows. 
\end{proof}

Applying Lemma \ref{lem:periodicorbits} to \cite[Theorem~1]{LMW2023a}, we have the following.
\begin{prop}
  \label{prop:effectiveequidistribution}
  Let $L\geq 1$ be an integer, $g_0 = \left(  \begin{pmatrix}  x_j^{\frac{1}{2}} & 0 \\
    0 & x_j^{-\frac{1}{2}}
  \end{pmatrix}
  \right)_{\!\! j=1,2}$ with $\tfrac{x_1}{x_2}$ $\kappa$-Diophantine, and $\varphi$ be a smooth compactly supported function on $\Gamma^2 \backslash G^2$.
  Then there exists $c > 0$ such that
  \begin{equation}
    \label{eq:effectiveequidistribution}
    \tfrac{1}{L} \int_0^L \varphi(g_0n(\xi) a(t)) \dd \xi = \int_{\Gamma^2 \backslash G^2} \varphi(g) \dd \mu(g) + O( \cS(\varphi) \e^{-c t})
  \end{equation}
  where $\mu$ is the Haar probability measure on $\Gamma^2 \backslash G^2$ and $\cS(\varphi)$ is a certain Sobolev norm of $\varphi$. 
\end{prop}

\begin{proof}
  We consider the integral $\int_{0}^1 \varphi(g_0 n(l + \xi) a(t)) \dd \xi$ for $l = 0,\dots, L-1$.
  Let $R = \e^{c_1 t}$ for small $c_1 > 0$.
  By \cite[Theorem~1]{LMW2023a}, we have either
  \begin{equation}
    \label{eq:effectiveequidist1}
    \int_{0}^1 \varphi(g_0 n(l + \xi) a(t)) \dd \xi = \int_{\Gamma^2 \backslash G^2} \varphi(g) \dd \mu (g) + O(\cS(\varphi) R^{-c_2})
  \end{equation}
  or there is $g \in \Gamma^2 \backslash G^2$ such that $\operatorname{Vol}(g H) \leq R$ and
  \begin{equation}
    \label{eq:effectiveequidistalt}
    \operatorname{dist}_{\Gamma^2 \backslash G^2} ( g_0n(l + \xi), gH) \leq (t R)^A \e^{-t}.
  \end{equation}
  By taking $c_1$ to be sufficiently small (depending only on $\kappa$), Lemma \ref{lem:periodicorbits} implies that (\ref{eq:effectiveequidistalt}) cannot hold.
  Therefore (\ref{eq:effectiveequidist1}) holds with remainder bound $\cS(\varphi) \e^{-c t}$ with $c = c_1 c_2$. 
\end{proof}

\section{Asymptotics of $\cN_{M,L}(F,\psi) $}
\label{sec:asymptotics}

In this section we determine the asymptotic behavior of $\cN_{M,L}(F,\psi)$ assuming Theorem~\ref{thm:escapeofmass}, which is proved in the following sections.
We recall
\begin{equation}
  \label{eq:RLMtheta}
  \cN_{M,L}(F,\psi)  = \frac{M^d}{2L} \int_{\R} \hat{\psi}(\tfrac{1}{2L} \xi) \widetilde\Theta_F(1, g_0 n(\xi) a(t) g_0^{-1}) \dd \xi
\end{equation}
with $t = \log M$.

\subsection{Small $\xi$}
\label{sec:smallxi}

We apply the following proposition with $\epsilon = \e^{-\eta t}$, $\eta > 0$ small, to obtain the first term in \eqref{eq:main2}.
\begin{prop}
Let $d\geq 3$. For $\epsilon>0$ and $A > 1$,
\begin{align}\label{lemX1} 
& \tfrac{1}{2L} \int_{|\xi|\leq \epsilon \e^{-t}} \hat{\psi} ( \tfrac{1}{2L}\xi ) \widetilde\Theta_F(1, g_0 n(\xi) a(t) g_0^{-1}) \dd \xi \notag \\
& = \tfrac{1}{2L} \e^{(d-2)t}  \hat \psi(0)  \int_\R \bigg(\int_{\mathbb{R}^{2d}} F(v) \e( \tfrac{1}{2} \xi Q(v) ) \dd v \bigg) \dd\xi \\
& \qquad + O_{F,\psi, A}(\epsilon^{A} \e^{dt} + L^{-2} \e^{(d-4)t} + L^{-1} \epsilon^{-d + 1} \e^{-t}) . \notag
\end{align}
\end{prop}

\begin{proof}
It will be convenient to write the theta function as 
\begin{equation}
\widetilde\Theta_F(1, g_0 n(\xi) a(t) g_0^{-1}) =
\widetilde\Theta_{F^0}(1, g_0 n(\xi) a(t))
\end{equation}
with $F^0 = R_{d,d}(g_0^{-1}) F$.
We set 
\begin{equation}
  \label{eq:JgXM}
  g = \left(  \begin{pmatrix}   0 & -1 \\
    1 & 0 
  \end{pmatrix}
  \right)_{\!\!1\leq j\leq d}
  g_0  n(\xi) a(t)  = \left(  \begin{pmatrix}   0 & - \e^t x_j^{-\frac{1}{2}} \\
    \e^{-t} x_j^{\frac{1}{2}} & \xi \e^t x_j^{\frac{1}{2}}
  \end{pmatrix}
  \right)_{\!\!1\leq j\leq d},
\end{equation}
and we compute that
\begin{equation}
  \label{eq:gIwasawa}
  g = \left(  \begin{pmatrix}   1 & u_j \\
    0 & 1
  \end{pmatrix}
  \begin{pmatrix}
    v_j^{\frac{1}{2}} & 0 \\
    0 & v_j^{-\frac{1}{2}}
  \end{pmatrix}
  \begin{pmatrix}
    \cos \theta_j & -\sin \theta_j \\
    \sin \theta_j  & \cos \theta_j
  \end{pmatrix}
  \right)_{\!\!1\leq j\leq d},
\end{equation}
with
\begin{equation}
  \label{eq:YRS}
  v_j =  (\e^{-2t} + \xi^2\e^{2t})^{-1}  x_j^{-1},\ \cos \theta_j = \xi \e^t x_j^{\frac{1}{2}}v_j^{\frac{1}{2}},\  \sin \theta_j = \e^{-t} x_j^{\frac{1}{2}}v_j^{\frac{1}{2}}. 
\end{equation}
We observe that if $|\xi| \leq \epsilon \e^{-t}$, then
\begin{equation}
  \label{eq:vnGammang}
  \min_{1\leq j\leq d} v_j \gg \epsilon^{-2}.
\end{equation}
By Lemma \ref{lem:shortestvector}, we may choose $\epsilon$ sufficiently small so that $\alpha_d(g) = \rho_d(g)^{-1}$ and $\alpha_{d-1}(g) \ll \epsilon \alpha_d(g)$. 

By Theorem \ref{thm:thetaasymptotic} and $v_j \ll e^{2t}$, we have
\begin{equation}
  \label{eq:Thetaasymptotic}
  \widetilde\Theta_{F^0}( 1, g_0 n(\xi) a(t ))  = ( v_1 \cdots v_d)^{\frac{1}{2}} F^0_\theta(0) + O_{F, A}(\epsilon^A\e^{dt})
\end{equation}
for any $A > 0$ and from Proposition \ref{prop:oscillator} and \eqref{eq:YRS}, we have
\begin{equation}
  \label{eq:fQ0}
  F^0_\theta(0) = \e^{dt} (x_1 \cdots x_d)^{-\frac{1}{2}}(v_1 \cdots v_d)^{-\frac{1}{2}} \int_{\mathbb{R}^{2d}} F^0(u,u') \e( \tfrac{1}{2} \xi \e^{2t} (u\transp{u}-u'\transp{u}') ) \dd u \dd u'.
\end{equation}
It follows that
\begin{multline}
  \label{eq:smallxi}
  \tfrac{1}{2L} \int_{|\xi|\leq \epsilon \e^{-t}} \hat{\psi} ( \tfrac{1}{2L}\xi ) \widetilde\Theta_{F^0}(1, g_0 n(\xi) a(t) ) \dd \xi \\
  = \tfrac{1}{2L}\e^{dt} (x_1 \cdots x_d)^{-\frac{1}{2}} \int_{|\xi| \leq \epsilon \e^{-t}} \hat{\psi}( \tfrac{1}{2L}\xi)  \bigg( \int_{\mathbb{R}^{2d}} F^0(u,u') \e( \tfrac{1}{2} \xi \e^{2t} (u\transp{u}-u'\transp{u}') ) \dd u \dd u'\bigg)  \dd \xi \\
  + O_{F, \psi, A}(\epsilon^A\e^{dt}).
\end{multline}
The inner integral over $\dd u \dd u'$ is $O_{F}( (1 + |\xi| \e^{2t})^{-d})$, so replacing $\hat \psi(\tfrac{1}{2L} \xi)$ with $\hat \psi(0)$ creates an error bounded by
\begin{equation}
  \label{eq:psi0replace}
  \ll_{\psi, F} L^{-2} \e^{dt} \int_{|\xi| \leq \epsilon \e^{-t}} \frac{ |\xi|}{(1 + |\xi| \e^{2t})^d} \dd \xi \ll L^{-2} \e^{(d-4)t}. 
\end{equation}
After changing variables $\xi \gets \e^{-2t} \xi$, we have that (\ref{eq:smallxi}) is 
\begin{equation}
\begin{split}
  \label{eq:smallxi1}
&\tfrac{1}{2L} \e^{(d-2)t} (x_1 \cdots x_d)^{-\frac{1}{2}}\hat\psi(0)  \int_{|\xi| \leq \epsilon \e^{t}} \bigg( \int_{\mathbb{R}^{2d}} F^0(u,u') \e( \tfrac{1}{2} \xi (u\transp{u}-u'\transp{u}') ) \dd u \dd u'\bigg)  \dd \xi \\
&  \qquad + O_{F,\psi,A}(\epsilon^{A}\e^{dt} + L^{-2} \e^{(d-4)t})\\
& = \tfrac{1}{2L} \e^{(d-2)t} \hat \psi(0)  (x_1 \cdots x_d)^{-\frac{1}{2}} \int_\R \bigg( \int_{\mathbb{R}^{2d}} F^0(u,u') \e( \tfrac{1}{2} \xi (u\transp{u}-u'\transp{u}') ) \dd u \dd u'\bigg)  \dd \xi \\
 & \qquad + O_{F,\psi, A}(\epsilon^{A}\e^{dt} + L^{-2} \e^{(d-4)t} + L^{-1} \epsilon^{-d + 1} \e^{-t}).
\end{split}
\end{equation}

Finally, recall that $F^0=R_{d,d}(g_0^{-1})F$, and set \(X_0=\operatorname{diag}(x_1,\ldots,x_d)\). After the corresponding change of variables, we obtain
\begin{equation}
\begin{split}
 & (\det X_0)^{-\frac{1}{2}} \int_{\mathbb{R}^{2d}} F^0(u,u') \e( \tfrac{1}{2} \xi (u\transp{u}-u'\transp{u}') ) \dd u \dd u' \\
& = (\det X_0)^{-1} \int_{\mathbb{R}^{2d}} F(u X_0^{-1/2},u' X_0^{-1/2}) \e( \tfrac{1}{2} \xi (u\transp{u}-u'\transp{u}') ) \dd u \dd u' \\
& = \int_{\mathbb{R}^{2d}} F(v) \e( \tfrac{1}{2} \xi Q(v) ) \dd v.
\end{split}
\end{equation}
\end{proof}

\subsection{Remaining $\xi$}
\label{sec:remaining}

To analyze the integral \eqref{eq:RLMtheta} over $|\xi| \geq \e^{-(1 + \eta)t}$, we appeal to results on the equidistribution of unipotent flows, i.e. Ratner's theorems and their relatives.

We let $\Gamma' \subset \Gamma^d $ be the finite index subgroup of $\gamma = \left(
\begin{pmatrix}
  a_j & b_j \\
  c_j & d_j
\end{pmatrix}
\right)_{\!\!1\leq j\leq d}$ such that $c_jd_j \equiv a_jb_j \equiv 0 \bmod 2$.
We note that from Theorem~\ref{thm:thetaautomorphy}, $\widetilde\Theta_F(1, g)$ is a well-defined function on $\Gamma' \backslash G^d$.

\begin{lem}
  \label{lem:density}
  Let $x = (x_j)_j \in \R_{>0}^d$ be so that $x_i/x_j\notin\Q$ for all $i\neq j$, and let $$g_0 = \left(  \begin{pmatrix}   x_j^{\frac{1}{2}} & 0 \\
    0 & x_j^{-\frac{1}{2}}
  \end{pmatrix}
  \right)_{\!\!1\leq j\leq d} \in G^d.$$
  Then $\Gamma' g_0 H g_0^{-1}$ is dense in $\Gamma' \backslash G^d$.
\end{lem}

\begin{proof}
The connected subgroups of $G^d$ containing $H$ are of the form 
$$
H_{\mathcal P}=\{ (g_1,\ldots,g_d) \in G^d : \text{$g_i=g_j$ if $i,j\in B_r$ for some $r$}\}
$$
where ${\mathcal P}=\{ B_1,\ldots, B_k\}$ runs through all partitions of $\{1,\ldots d\}$. If the closure of $\Gamma' g_0 H g_0^{-1}$ is $\Gamma' g_0 H_{\mathcal P} g_0^{-1}$, then factors of the discrete subgroups corresponding to the same $B_r$ have to be commensurable. This implies that, for $i,j\in B_r$, 
$$
\  \begin{pmatrix}   (x_i/x_j)^{\frac{1}{2}} & 0 \\
    0 & (x_i/x_j)^{-\frac{1}{2}} 
 \end{pmatrix} 
    \SL(2,\Z) 
    \begin{pmatrix}   (x_i/x_j)^{\frac{1}{2}} & 0 \\
    0 & (x_i/x_j)^{-\frac{1}{2}}
    \end{pmatrix}^{-1}  
$$
is commensurable with $\SL(2,\Z)$; this in turn is equivalent to $x_i/x_j\in\Q$. By assumption $x_i/x_j\notin\Q$ for $i\neq j$, so the only permitted partition is by singletons, and hence $H_{\mathcal P}=G^d$.
\end{proof}

By a theorem of Shah \cite{MR1403756}, Lemma \ref{lem:density} implies that for any bounded continuous function $\varphi$ on $\Gamma' \backslash G^d$, 
\begin{equation}
  \label{eq:equidistribution}
  \lim_{t\to \infty} \tfrac{1}{2L} \int \hat\psi( \tfrac{\xi}{2L}) \varphi(g_0n(\xi)a(t)) \dd \xi = \psi(0)\int_{\Gamma' \backslash G^d} \varphi(g) \dd \mu(g),
\end{equation}
where $\mu$ is the $G^d$-invariant probability measure on $\Gamma' \backslash G^d$. Shah's theorem only holds for fixed $L$, but the convergence holds in fact uniformly for $L\geq 1$ by the argument used in the proof of Theorem of 5.1 of \cite{Marklof2002}.

We now use Theorem~\ref{thm:escapeofmass} (proved below) to extend \eqref{eq:equidistribution} to the unbounded theta function. We specialize to $d = 3$.

Let $\cC(y) = \{g \in \Gamma^3 \backslash G^3 : \alpha_3(g) \geq y\}$, and let $\chi_y$ be a continuous function on $\Gamma' \backslash G^3$ satisfying $\mathbbm{1}_{\cC(2y)} \leq \chi_y \leq \mathbbm{1}_{\cC(y)}$.
We deduce from Theorem \ref{thm:thetaasymptotic} that
\begin{multline}
\label{eq:tailestimate0}
  \tfrac{1}{2L} \bigg| \int_{|\xi| \geq \e^{-(1 + \eta)t}} \hat\psi(\tfrac{\xi}{2L})|\widetilde\Theta_F(1, g_0n(\xi) a(t))| \, \chi_y(g_0n(\xi)a(t)) \dd \xi \bigg| \\
  \ll \tfrac{1}{2L} y^{1 - \beta} \int_{|\xi| \geq \e^{-(1+\eta)t}} \big| \hat\psi(\tfrac{\xi}{2L}) \big| \alpha_d(g_0n(\xi) a(t))^\beta \dd \xi \end{multline}
for some $\beta > 1$.
We let $\theta > 0$ be as in Theorem \ref{thm:escapeofmass} and bound \eqref{eq:tailestimate0} by
\begin{multline*}
  \tfrac{1}{2L} y^{1 - \beta} \int_{\e^{-(1+\eta)t} \leq |\xi| \leq L} \big| \hat\psi(\tfrac{\xi}{2L})\big| \alpha_d(g_0n(\xi) a(t))^\beta \dd \xi + \tfrac{1}{2L}y^{1-\beta}\int_{|\xi| > Le^{\theta t}} \big| \hat\psi(\tfrac{\xi}{2L})\big| \alpha_d(g_0n(\xi) a(t))^\beta \dd \xi \\
  + \tfrac{1}{2L} y^{1-\beta}\sum_{0\leq k \leq K} \int_{2^k L \leq |\xi| \leq 2^{k+1} L} \big| \hat\psi(\tfrac{\xi}{2L})\big| \alpha_d(g_0n(\xi) a(t))^\beta \dd \xi
\end{multline*}
with $K \ll t $.
By Theorem \ref{thm:escapeofmass} the first integral is $O_\psi(y^{1-\beta}(1 + L^{-1} \e^{(1-\tau)t}))$ for some $0 < \tau < \tfrac{1}{\kappa + 1}$, while by the decay of $\hat \psi$, i.e., for every $A>0$,
\[
  \left|\hat\psi\left(\frac{\xi}{2L}\right)\right|
  \ll_{\psi,A}
  \left(1+\frac{|\xi|}{L}\right)^{-A},
\] as well as the bound $\alpha_3(g_0 n(\xi) a(t)) \ll \e^{3t}$, the second integral also satisfies this bound.
For the third expression, we combine Theorem \ref{thm:escapeofmass} with the decay of $\hat \psi$ to obtain the bound
\begin{equation*}
  y^{1-\beta} \sum_{0\leq k \leq K} 2^{-kA} ( 2^{k+1} + L^{-1} \e^{(1-\tau)t}) \ll y^{1-\beta} ( 1 + L^{-1} \e^{(1-\tau)t}). 
\end{equation*}
We therefore have
\begin{multline}
  \label{eq:tailestimate}
  \tfrac{1}{2L} \bigg| \int_{|\xi| \geq \e^{-(1 + \eta)t}} \hat\psi(\tfrac{\xi}{2L})|\widetilde\Theta_F(1, g_0n(\xi) a(t))| \, \chi_y(g_0n(\xi)a(t)) \dd \xi \bigg| \ll y^{1-\beta}( 1 + \tfrac{1}{L} \e^{(1-\tau)t}).
\end{multline}
We note that, by taking $y$ a constant, \eqref{eq:tailestimate} implies \eqref{eq:main2} in the range $L \leq M^{1 - \tau}$ when the second term and $o(M^3)$ term are dominated by $O(\tfrac{1}{L} M^{4 -\tau})$.
For the remainder of this section, we may assume $L > M^{1-\tau}$.

On the support of $1-\chi_y$ we have $\alpha_3(g)<2y$.
Therefore, by Theorem~\ref{thm:thetaasymptotic},
\[
|\widetilde\Theta_F(1,g)|(1-\chi_y(g))\ll_F y.
\]
Hence the decay of $\hat \psi$ and \eqref{eq:equidistribution} imply that
\begin{equation}
\begin{split}
  \label{eq:equidistribution1}
&  \tfrac{1}{2L}
\int_{|\xi| \geq \e^{-(1+\eta)t}}
\hat\psi(\tfrac{\xi}{2L})
\widetilde\Theta_F(1,g_0n(\xi)a(t))
(1-\chi_y(g_0n(\xi)a(t)))\,\dd\xi \\
&  =
\tfrac{1}{2L}\int_{|\xi| \leq L e^{\theta t}}
\hat\psi(\tfrac{\xi}{2L})
\widetilde\Theta_F(1,g_0n(\xi)a(t))
(1-\chi_y(g_0n(\xi)a(t)))\,\dd\xi
+O_{F,\psi}\left(\tfrac{y}{L}\e^{-t}\right)
+o_{F,\psi,y}(1)\\
&  \underset{t\to\infty}{\to} \psi(0) \int_{\Gamma' \backslash G^d} \widetilde\Theta_F(1, g) ( 1 - \chi_y(g)) \dd \mu(g) \\
&  = \psi(0) \int_{\Gamma' \backslash G^d} \widetilde\Theta_F(1, g) \dd \mu(g) + O(y^{-1}(\log y)^2),
\end{split}
\end{equation}
where the last line follows from $\mu(\cC(y) \setminus \cC(2y)) \ll y^{-2}(\log y)^2$. 

Combining this with \eqref{eq:tailestimate}, we obtain
\begin{equation}
  \label{eq:ratner1}
  \lim_{t \to \infty} \tfrac{1}{2L} \int_{|\xi|\geq\epsilon \e^{-t}} \hat\psi(\tfrac{1}{2L} \xi) \widetilde\Theta_F(1, g_0n(\xi)a(t)) \dd \xi \\
  = \psi(0) \int_{\Gamma' \backslash G^d} \widetilde\Theta_F(1, g) \dd\mu(g). 
\end{equation}
Together with Lemma \ref{lem:thetaint} below, this gives the second term in \eqref{eq:main2}.

\begin{lem}
  \label{lem:thetaint}
  For $F\in \cS(\R^{2d})$, $d\geq 1$,
  \begin{equation}
    \label{eq:thetaint}
    \int_{\Gamma' \backslash G^d} \widetilde\Theta_F(1, g) \dd \mu(g) = \sum_{\sigma\in D_{2h}} \int_{\R^d} F(x,\sigma(x)) \dd x.
  \end{equation}
\end{lem}

\begin{proof}
  As $\tilde n$ parametrizes the full horospherical subgroup associated to $a(t)$, it follows from the equidistribution of long closed horocycles on each factor that
  \begin{equation}
    \label{eq:fullhorosphere}
    \int_{\Gamma' \backslash G^d} \widetilde\Theta_F(1, g) \dd \mu(g) = \lim_{t\to \infty} 2^{-d} \int_{[0,2]^d} \widetilde\Theta_F(1, \tilde n(\xi)a(t)) \dd \xi.
  \end{equation}
  The theta functions are unbounded, but convergence still holds by standard escape of mass estimates on each factor \cite{Marklof1999}.
  From \eqref{eq:thetaexdd} we have
  \begin{equation}
  \begin{split}
    \label{eq:horosphereaverage}
    & 2^{-d} \int_{[0,2]^d} \widetilde\Theta_F(1, \tilde n(\xi)a(t)) \dd \xi \\
    & = 2^{-d}\e^{-dt}\sum_{m_1, m_2 \in \Z^d} F(m_1\e^{-t}, m_2e^{-t}) \int_{[0,2]^d} \e(\tfrac{1}{2} \sum_{1\leq j\leq d} \xi_j (m_{1,j}^2 - m_{2,j}^2)) \dd \xi \\
    & = \e^{-dt}\sum_{m \in \Z^d}\sum_{\sigma\in D_{2h}} F(m\e^{-t},\sigma(m) e^{-t}) + O_F(\e^{-(d+1)t}),
  \end{split}
\end{equation}
where the remainder bound controls the over counting of $m$ with some coordinates $0$. 
  As $t\to \infty$ this Riemann sum converges to the right side of \eqref{eq:thetaint}.
\end{proof}

The statement of the lemma also follows from representation theory, cf.~Eqs.~(41) and (48) in \cite{Marklof1999}.

\section{Escape of mass}
\label{sec:escape}

In this section we prove Theorem \ref{thm:escapeofmass} assuming the avoidance estimate (Proposition~\ref{prop:avoidanceMargulis}) proved in Section \ref{sec:avoidance}.

From $g_0 n(\xi) a(t) = g_0 a(-t_0) n(e^{-2t_0} \xi) a(t + t_0)$, we observe that multiplying all the coefficients $x_j$ by $\e^{t_0}$ only rescales the variables $\e^t$ and $L$ by fixed factors, and does not change any of the ratios $x_i/x_j$ or their Diophantine properties. We may therefore normalize the coefficients so that
\begin{equation}
  \label{eq:xnormalization}
  \min_{1\leq j\leq3}x_j\geq1.
\end{equation}

\subsection{Local contraction}
\label{sec:loccontraction}

The main result of this section is Lemma \ref{lem:phicontraction}, which establishes a local contraction property of the height functions $\rho_k$, compare to \cite[Lemma~5.1]{KleinbockMargulis1999}.
In our setting, we modify these height functions to also measure the distance to an exceptional set on which $\rho_k^{-\beta}$ does not contract for $\beta > \tfrac{2}{k}$.
In this way, we obtain contraction for the modified functions for $1 < \beta < 2$.

The distance to the exceptional set is controlled by the following functions.
For $$g =
\left(
\begin{pmatrix}
  a_j & b_j \\
  c_j & d_j
\end{pmatrix}
\right)_{\!\! j=1,2,3} \in G^3$$ we define
\begin{equation}
  \label{eq:Delta2def}
  \Delta_2(g) = \min_{j = 1,2,3} \max_{i \neq j} \min ( 1, | \tfrac{d_i}{c_i} - \tfrac{d_j}{c_j}|)
\end{equation}
and
\begin{equation}
  \label{eq:Deltadef}
  \Delta_3(g) = \min_{\!\! j=1,2,3} \prod_{i \neq j} \min(1, |\tfrac{d_i}{c_i} - \tfrac{d_j}{c_j}|),
\end{equation}
where, here and in what follows, if $c_ic_j = 0$, we take $1$ in the corresponding minimum.
The functions $\Delta_2$ and $\Delta_3$ control the coincidence between sets of three and two roots, respectively, of the polynomial $\prod_{1\leq j \leq 3} (\xi + \tfrac{d_j}{c_j})$ and appear naturally in sublevel estimates of $\rho_3(g n(\xi) a(s))$, see Lemma \ref{lem:measbound}. 
We note that enforcing the maximum value of $\Delta_2$, $\Delta_3$ to be $1$ appears necessary to handle cases when one or more $c_j$ is small, a particularly relevant case under translation by $a(s)$, see Lemma \ref{lem:phicontraction}. 

\begin{lem}
  \label{lem:measbound}
  Let $g = \left( \begin{pmatrix}    r_j & -s_j \\
    s_j & r_j
  \end{pmatrix}
  \right)_{\!\! j=1,2,3} \in \SO(2)^3$.
  For any $s > 0$, the measure of the set of $\xi \in [-1,1]$ such that $\rho_3(gn(\xi)a(s)) \leq \delta$ is $O(\min( \delta^{\frac{1}{3}}\e^{-s} , \delta \Delta_2 (g)^{-1} \e^{-s}, \delta \Delta_3(g)^{-1} \e^{-3s}))$. 
\end{lem}

\begin{proof}
  We note that since $r_j^2 + s_j^2 = 1$, if $s_j^2 \leq \tfrac{1}{10}$, we have $|r_j + s_j \xi| \asymp 1$ for $\xi \in [-1,1]$.
  Therefore
  \begin{equation}
  \begin{split}
    \label{eq:rhokgna}
    \rho_3 (gn(\xi) a(s)) & = \prod_{j = 1}^3 |(\e^{-s} s_j, \e^s(r_j + s_j \xi))| \\
    &\gg \e^{3s} \prod_{s_j^2 > \tfrac{1}{10}} |r_j + s_j \xi| \gg \e^{3s} \prod_{s_j^2 > \tfrac{1}{10}} |\xi + \tfrac{r_j}{s_j}|,
  \end{split}
  \end{equation}
  and we may assume that product is nonempty since otherwise the lemma is trivial.

  Let $j$ be so that $|\xi + \tfrac{r_j}{s_j}|$ is minimal among $j$ so that $s_j^2 > \tfrac{1}{10}$.
  If $|\tfrac{r_j}{s_j}| > 2$, then the lemma is again trivial.
  It follows that (\ref{eq:rhokgna}) is
  \begin{equation}
    \label{eq:rho3cubelowerbound}
    \gg \e^{3s} |\xi + \tfrac{r_j}{s_j}|^3. 
  \end{equation}
  We also have that (\ref{eq:rhokgna}) is
  \begin{equation}
    \label{eq:rho3linearlowerbound}
    \gg \e^{3s} |\xi + \tfrac{r_j}{s_j}| \prod_{\substack{s_i^2 > \tfrac{1}{10} \\ i \neq j}} | \tfrac{r_i}{s_i} - \tfrac{r_j}{s_j} | \gg \e^{3s} \Delta_3(g) |\xi + \tfrac{r_j}{s_j}|, 
  \end{equation}
  where the last bound follows from $| \tfrac{r_i}{s_i} - \tfrac{r_j}{s_j} |\gg 1$ if $s_i^2 \leq \tfrac{1}{10}$.
  
  The first and third bounds claimed in the lemma follow from (\ref{eq:rho3cubelowerbound}) and (\ref{eq:rho3linearlowerbound}).
  For the second claimed bound, we aim to show
  \begin{equation}
    \label{eq:rho3quadlowerbound1}
    \rho_3(gn(\xi)a(s)) \gg \e^s \Delta_2(g) |\xi + \tfrac{r_j}{s_j}|
  \end{equation}
  with $j$ as before.  

  Without loss of generality we assume $j = 3$.
  If both $s_1^2 \leq \tfrac{1}{10}$ and $s_2^2 \leq \tfrac{1}{10}$, then $\Delta_2(g) = \Delta_3(g) = 1$, so (\ref{eq:rho3quadlowerbound1}) follows from (\ref{eq:rho3linearlowerbound}).
  If $s_1^2 \leq \tfrac{1}{10} < s_2^2$, then from the left of (\ref{eq:rhokgna}) we have
  \begin{equation}
    \label{eq:rho3quadlowerbound4}
    \rho_3(gn(\xi)a(s)) \geq \e^s |s_2| |r_1 + s_1 \xi| |r_3 + s_3 \xi| \gg \e^s |\xi + \tfrac{r_3}{s_3}|,
  \end{equation}
  and (\ref{eq:rho3quadlowerbound1}) follows similarly if $s_2^2 \leq \tfrac{1}{10} < s_1^2$.

  Finally, if both $s_1^2, s_2^2 > \tfrac{1}{10}$, we assume without loss of generality that $|\tfrac{r_1}{s_1} - \tfrac{r_3}{s_3}| \geq |\tfrac{r_2}{s_2} - \tfrac{r_3}{s_3}|$.
  We have
  \begin{equation}
    \label{eq:rho3quadlowerbound5}
    \rho_3(gn(\xi)a(s)) \geq \e^s |s_2| |r_1 + s_1 \xi| |r_3 + s_3 \xi| \gg \e^s |\tfrac{r_1}{s_1} - \tfrac{r_3}{s_3}||\xi + \tfrac{r_3}{s_3}|,
  \end{equation}
  so (\ref{eq:rho3quadlowerbound1}) follows. 
\end{proof}

\begin{lem}
  \label{lem:rhocontraction}
  For any $s > 0$, $\beta > 1$, and $g \in G^3$, we have
  \begin{equation}
    \label{eq:loccontraction}
    \int_{-1}^1 \rho_3 (g n(\xi) a(s))^{-\beta} \dd \xi 
    \ll \min (\e^{(3\beta -2)s}, \Delta_2(g)^{-1} \e^{(3\beta - 4)s}, \Delta_3(g)^{-1} \e^{(3\beta -6)s}) \rho_3(g)^{-\beta}.
  \end{equation}
\end{lem}

\begin{proof}
  We note that if $g' = \left(
  \begin{pmatrix}
    * & * \\
    0 & * 
  \end{pmatrix}\right)_{\!\! j=1,2,3}$, then $\rho_3(g'g) = \rho_3(g')\rho_3(g)$, so we may assume $g =
  \left(
  \begin{pmatrix}
    r_j & -s_j \\
    s_j & r_j
  \end{pmatrix}\right)_{\!\! j=1,2,3}
  \in \SO(2)^3$, where $\rho_3(g) = 1$.
  We note that since $r_j^2 + s_j^2 = 1$, it follows that if $s_j^2 \leq \tfrac{1}{10}$, then $|r_j + s_j \xi| \gg 1$ for all $\xi \in [-1,1]$.
  It follows that $\rho_3(gn(\xi)a(s)) \gg \e^{-(3 - 2k)s} \geq \e^{-3s}$, where $k$ is the number of $j$ such that $s_j^2 \leq \tfrac{1}{10}$.
  We also note that $\rho_3(gn(\xi)a(s)) \ll \e^{3s}$. 

  We decompose the region of integration according to $\xi$ such that $\delta < \rho_3(gn(\xi) a(s)) \leq 2 \delta$ ranges dyadically from $\e^{-3s}$ to $\e^{3s}$ ($\delta = c 2^j \e^{-3s}$, $0 \leq j \ll s$, and $c > 0$ a small constant).
  We obtain
  \begin{equation}
  \begin{split}
    \label{eq:dyadicdecomp}
    \int_{-1}^1 \rho_3 (g n(\xi) a(s))^{-\beta} \dd \xi 
&    \ll \sum_{\substack{ \delta \gg \e^{-3s} \\ \mathrm{dyadic}}} \delta^{-\beta} \mathrm{meas}\{ \xi \in [-1,1]: \rho_3(gn(\xi)a(s)) \leq \delta\} \\
 &   \ll \sum_{\substack{ \delta \gg \e^{-3s} \\ \mathrm{dyadic}}} \delta^{-\beta} \min( \delta^{\frac{1}{3}} \e^{-s}, \delta \Delta_2(g)^{-1} \e^{-s}, \delta \Delta_3(g)^{-1} \e^{-3s}) \\
 &   \ll \min (\e^{(3\beta -2)s}, \Delta_2(g)^{-1} \e^{(3\beta - 4)s}, \Delta_3(g)^{-1} \e^{(3\beta -6)s}),
  \end{split}
  \end{equation}
  the second bound following from Lemma \ref{lem:measbound}.
\end{proof}

We set
\begin{equation}
  \label{eq:phidef}
  \phi(g) = \Delta_3(g)^{-1} \rho_3(g)^{-\beta}
\end{equation}
with $\phi(g) = +\infty$ if $\Delta_3(g) = 0$.
\begin{lem}
  \label{lem:phicontraction}
  For any $g \in G^3$, $s > 0$, and $\beta > 1$, we have
  \begin{equation}
    \label{eq:phicontraction}
    \int_{-1}^1 \phi(gn(\xi)a(s)) \dd \xi \ll \e^{-3 (2-\beta) s} \phi(g).
  \end{equation}
\end{lem}

\begin{proof}
  We observe that 
      \begin{equation}
    \label{eq:kappaas}
    \Delta_3(gn(\xi)a(s))^{-1}\Delta_3(g)
    =
    \begin{cases}
      \e^{-4s},
      & \Delta_2(g)\leq\e^{-2s},\\
      \e^{-2s}\Delta_2(g),
      & \dfrac{\Delta_3(g)}{\Delta_2(g)}
        \leq\e^{-2s}\leq\Delta_2(g),\\
      \Delta_3(g),
      & \e^{-2s}\leq\dfrac{\Delta_3(g)}{\Delta_2(g)}.
    \end{cases}
  \end{equation}
  Applying the bounds in Lemma \ref{lem:rhocontraction} to each of these cases (in the corresponding order) implies (\ref{eq:phicontraction}). 
\end{proof}

We end this section by recording the following lemma establishing a weaker bound than Lemma~\ref{lem:phicontraction}, but for a larger range of $\beta$ that we omitted from the proof of Lemma~\ref{lem:phicontraction} for the sake of exposition.
\begin{lem}
  \label{lem:cruderhocontraction}
  For any $d\geq 1$, $1\leq k\leq d$, $s > 0$, $\beta > \tfrac{1}{k}$, and $g \in G^d$, we have
  \begin{equation}
    \label{eq:crudeloccontraction}
    \int_{-1}^1 \rho_k (g n(\xi) a(s))^{-\beta} \dd \xi \ll \e^{-(2-\beta k )s} \rho_k(g)^{-\beta}.
  \end{equation}
\end{lem}

\begin{proof}
    For $J\subset\{1,\ldots,d\}$, let
\[
  \rho_J(g)=\prod_{j\in J}|(c_j,d_j)|,
\]
so that $\rho_k(g)=\min_{|J|=k}\rho_J(g)$. Write
$g_j=b_j\widetilde g_j$, where $b_j$ is upper triangular and
$\widetilde g_j\in\SO(2)$. For each fixed $J$ we have
\[
  \rho_J(gn(\xi)a(s))
  =
  \rho_J(g)\rho_J(\widetilde g n(\xi)a(s)).
\]
Since $\rho_J(g)\geq\rho_k(g)$, it follows that
\[
  \rho_k(gn(\xi)a(s))
  \geq
  \rho_k(g)\rho_k(\widetilde g n(\xi)a(s)).
\]
Thus it suffices to prove the estimate for
\[
  g=\left(
  \begin{pmatrix}
    r_j&-s_j\\
    s_j&r_j
  \end{pmatrix}
  \right)_{\!\!1\leq j\leq d}
  \in\SO(2)^d,
\]
where we have $\rho_k(g)=1$.

  We have
  \begin{equation}
    \label{eq:deltalowerbound}
    \rho_k\big(gn(\xi)a(s)\big) \asymp \min_{\substack{J \subset\{1,\dots, d\} \\ |J|=k}} \max_{J_1 \subset J} \e^{-(k-2|J_1|)s}\prod_{j \in J \setminus J_1} |s_j| \prod_{j \in J_1} |r_j + s_j \xi|.
  \end{equation}
  Let $J_0$ be the set of $j$ such that $s_j^2 \leq \tfrac{1}{10}$, so that if $j \in J_0$, $|r_j + s_j \xi| \asymp 1$ for $\xi \in [-1,1]$.
  Taking $J_1 = J_0 \cap J$ in (\ref{eq:deltalowerbound}) shows that $\rho_k(gn(\xi)a(s)) \gg \e^{-(k-|J_0 \cap J|)s} \geq \e^{-ks}$.

  On the other hand, we claim that the measure of $\xi \in [-1,1]$ such that $\rho_k(gn(\xi)a(s)) \leq \delta $ is $\ll \delta^{\frac{1}{k}} \e^{-s}$.
  For each $\xi$, choose $J$ attaining the minimum in
\eqref{eq:deltalowerbound}. Taking $J_1=J$, we have
  \begin{equation}
    \label{eq:rhoklowerbound}
    \rho_k(gn(\xi)a(s)) \gg \e^{ks} \prod_{j \in J \setminus J_0} |r_j + s_j \xi| \gg \e^{ks} \prod_{j \in J \setminus J_0} |\xi + \tfrac{r_j}{s_j}|. 
  \end{equation}
  Let $j\in J\setminus J_0$ be such that
$\left|\xi+\frac{r_j}{s_j}\right|$ is minimal. If $|\tfrac{r_j}{s_j}| \geq 2$ then the claim is trivial.
  Otherwise we have
  $$\rho_k(gn(\xi)a(s)) \gg \e^{ks} |\xi + \tfrac{r_j}{s_j}|^{|J \setminus J_0|} \gg \e^{ks} |\xi + \tfrac{r_j}{s_j}|^{k},$$
  and the claim follows in this case as well.

  Decomposing $\xi \in [-1,1]$ according to $\tfrac{1}{2} \delta < \rho_k(gn(\xi)a(s)) \leq \delta $, it follows that
 \begin{equation}
  \begin{split}
    \label{eq:dyadicdecomp2}
    \int_{-1}^1 \rho_k (g n(\xi) a(s))^{-\beta} \dd \xi 
& \ll \sum_{\substack{ \delta \gg \e^{-ks} \\ \mathrm{dyadic}}} \delta^{-\beta} \mathrm{meas}\{ \xi \in [-1,1]: \rho_k(gn(\xi)a(s)) \leq \delta\} \\
 &   \ll \sum_{\substack{ \delta \gg \e^{-ks} \\ \mathrm{dyadic}}} \delta^{\frac{1}{k}-\beta}\e^{-s}  \ll \e^{-(2 - \beta k)s},
  \end{split}
  \end{equation}
  as required. 
\end{proof}

\subsection{Global contraction}
\label{sec:globalcontraction}

We first prove a crude bound for averages of $\alpha_d(gn(\xi)a(s))^{\beta}$.
We note that this bound does not exhibit contraction for $\beta \geq \tfrac{2}{d}$.
\begin{lem}
  \label{lem:crudecontraction}
  Let $d\geq 1$. For any $g \in G^d$, $s > 0$, and $\beta > \tfrac{1}{d}$, we have
  \begin{equation}
    \label{eq:crudecontraction}
    \int_{-1}^1 \alpha_d(gn(\xi)a(s))^\beta \dd \xi 
    \ll \e^{(\beta d- 2)s} \alpha_d(g)^\beta + \e^{\beta s}\int_{-1}^1 \alpha_{d-1}(gn(\xi)a(s))^\beta \dd \xi .
  \end{equation}
  If, in addition, $\beta>1$, then
\begin{equation}
  \label{eq:crudecontraction1}
  \int_{-1}^1\alpha_d(gn(\xi)a(s))^\beta\dd\xi
  \ll
  \e^{(\beta d-2)s}\alpha_d(g)^\beta+\e^{\beta ds}.
\end{equation}
\end{lem}

\begin{proof}
  We let $\gamma(\xi) = \gamma(gn(\xi) a(s))$ and we note that \eqref{eq:crudecontraction} follows from the following
  \begin{equation}
    \label{eq:pointwisecrudecontraction}
    \rho_d(\gamma(\xi) gn(\xi)a(s))^{-\beta}  \ll \rho_d(\gamma(g) g n(\xi) a(s))^{-\beta} + \e^{ \beta s}\alpha_{d-1}(gn(\xi)a(s))^\beta
  \end{equation}
  together with an application of Lemma \ref{lem:cruderhocontraction} with $g \gets \gamma(g) g$. 

  We now establish \eqref{eq:pointwisecrudecontraction}.
  If $\gamma(\xi) = \left(\begin{pmatrix}    \pm 1 & * \\
    0 & \pm 1
  \end{pmatrix}
  \right)_{\!\!1\leq j\leq d} \gamma(g)$, then \eqref{eq:pointwisecrudecontraction} is immediate.
  If for some $j$, $\gamma_j(\xi) \neq \begin{pmatrix}    \pm 1 & * \\
    0 & \pm 1
  \end{pmatrix}
 \gamma_j(g)$, then it follows from Lemma \ref{lem:shortestvector} that $\rho(\gamma_j(\xi)g_j) > 1$, so by Lemma \ref{lem:rhotransbounds}
  \begin{equation}
    \label{eq:rhogammajxichange}
    \rho(\gamma_j(\xi)g_j n(\xi) a(s)) \geq \tfrac{1}{2} \e^{-s} \rho(\gamma_j(\xi)g_j) \geq \tfrac{1}{2} \e^{-s}.
  \end{equation}
  We now have
  \begin{equation}
    \label{eq:gammachangebound}
    \rho_d(\gamma(\xi)gn(\xi)a(s))^{-\beta} \ll \e^{\beta s} \alpha_{d-1}(gn(\xi)a(s))^\beta
  \end{equation}
  and \eqref{eq:pointwisecrudecontraction} follows in this case as well.

  Finally, we observe that \eqref{eq:crudecontraction1} follows from \eqref{eq:crudecontraction} inductively. When $\beta>1$, we may iterate \eqref{eq:crudecontraction}
through all lower ranks, since $\beta>1/k$ for every $1\leq k\leq d$.
This gives \eqref{eq:crudecontraction1}.
\end{proof}

In order to obtain a contraction, we define a global modified height function $\widetilde{\alpha}: \Gamma^3 \backslash G^3 \to (0,\infty]$ by
\begin{equation}
  \label{eq:tildealphadef1}
  \widetilde{\alpha}(g) = \phi(\gamma(g)g),
\end{equation}
where $\widetilde{\alpha}(g)=+\infty$ if
$\Delta_3(\gamma(g)g)=0$. We now consider exceptional sets to the global contraction hypothesis.
For $s>0$ and $K>1$ we define
\begin{multline}
  \label{eq:cEdef}
  \mathcal{E}_{s,K}:=\bigg\{\Gamma^3 g\in \Gamma^3 \backslash G^3: \Delta_3(\gamma(gn(\xi)a(s)) g)<\e^{-Ks} \\
  \textrm{ for some } 
  \xi \in [-1,1] \textrm{ such that }  \gamma(gn(\xi)a(s)) \neq \left(\begin{pmatrix}    \pm 1 & * \\
    0 & \pm 1
  \end{pmatrix}
  \right)_{\!\! j=1,2,3} \gamma(g)
  \bigg\}.
\end{multline}
The following global contraction hypothesis holds outside the set $\mathcal{E}_{s,K}$.
\begin{prop}
  \label{prop:globalcontraction}
  Let $1<\beta<2$.
  For any $s>0$, $K>1$, and $\Gamma^3 g \notin \mathcal{E}_{s,K}$ we have
  \begin{equation}
    \label{eq:globalcontraction}
    \int_{-1}^1\widetilde{\alpha}\big(gn(\xi)a(s)\big)\dd\xi \\
    \ll  e^{-3(2-\beta) s}\widetilde{\alpha}(g)+e^{(K+2)s}\int_{-1}^{1} \alpha_{2}(gn(\xi)a(s))^{\beta}\dd \xi.
  \end{equation}
\end{prop}

\begin{proof}
  We argue as in the proof of Lemma \ref{lem:crudecontraction}.
  It suffices to show that
  \begin{equation}
    \label{eq:pointwisecontractionbound}
    \phi(\gamma(\xi)gn(\xi)a(s)) \ll \phi(\gamma(g)gn(\xi)a(s)) +e^{(K+2)s} \alpha_{2}(gn(\xi)a(s))^{\beta},
  \end{equation}
  since then \eqref{eq:globalcontraction} follows from Lemma \ref{lem:phicontraction}.

  If $\gamma(\xi) = \left(\begin{pmatrix}    \pm 1 & * \\
    0 & \pm 1
  \end{pmatrix}
  \right)_{\!\! j=1,2,3} \gamma(g)$, then \eqref{eq:pointwisecontractionbound} is immediate, so let us suppose that there is $j$ such that $\gamma_j(\xi) \neq \begin{pmatrix}    \pm 1 & * \\
    0 & \pm 1
  \end{pmatrix}
  \gamma_j(g)$.
  Then we have $\rho(\gamma_j(\xi)g_jn(\xi)a(s)) \geq \tfrac{1}{2} \e^{-s}$ and therefore
  \begin{equation}
    \label{eq:gammachangebound1}
    \alpha_3(gn(\xi) a(s)) \leq 2\e^s \alpha_{2}(gn(\xi)a(s))
  \end{equation}
  as in the proof of Lemma \ref{lem:crudecontraction}.
  Since $g \not\in \cE_{s,K}$, \eqref{eq:pointwisecontractionbound} follows.
\end{proof}

In what follows it is at times easier to work with the following set.
  \begin{lem}
    \label{lem:altexceptionalset}
    We have
    \begin{multline}
      \label{eq:tildecEdef}
      \cE_{s,K} \subset \tilde{\cE}_{s,K} := \{\Gamma^3 g\in \Gamma^3 \backslash G^3: \Delta_3(\gamma g)<\e^{-Ks} \\
  \textrm{ for some } 
  \gamma \in \Gamma^3
  \textrm{ such that } \tfrac{1}{10}\e^{-s} < \max_{k=1,2,3}\rho(\gamma_kg_k) \leq 10 \e^{s} \}.
    \end{multline}
  \end{lem}
  
\begin{proof}
  As above,
  $\gamma_j(\xi) \neq \begin{pmatrix}    \pm 1 & * \\
    0 & \pm 1
  \end{pmatrix}
 \gamma_j(g)$ implies $\rho(\gamma_j(\xi)g_jn(\xi)a(s)) \geq \tfrac{1}{2} \e^{-s}$.
  We also have from Lemma \ref{lem:rhotransbounds} that
  \begin{equation}
    \label{eq:rhoupperbound}
    \rho(\gamma_j(\xi)g_j) \leq 2\e^s \rho(\gamma_j(\xi) g_j n(\xi)a(s)) \leq 10 \e^s,
  \end{equation}
  so the containment \eqref{eq:tildecEdef} follows.
\end{proof}

\subsection{Smaller moments}
\label{sec:smallmoments}

In this section we prove Proposition \ref{prop:smallmoment}, which controls the escape of mass for $\alpha_3^\beta$ with $\beta < \tfrac{2}{3}$.
Combined with Lemma \ref{lem:alpha3supbound} this allows for some useful preliminary control over the escape of mass.
\begin{prop}
  \label{prop:smallmoment}
  Let $d\geq 1$. For any $\Gamma^dg \in \Gamma^d \backslash G^d$ and $\beta < \tfrac{2}{d}$, we have
  \begin{equation}
    \label{eq:smallmoment}
    \sup_{t > 0} \int_{-10}^{10} \alpha_d(gn(\xi) a(t))^\beta \dd \xi < \infty.
  \end{equation}
\end{prop}

\begin{proof}
  We prove this by induction on $d$.
  For the case $d = 1$, we thicken the horocycle $gn(\xi)$ to the set $$\cS = \bigg\{ g n(\xi)
  \begin{pmatrix}
    a^{-1} & 0 \\
    a b & a
  \end{pmatrix}
  : |\xi| \leq 10,\ |1-a|\leq \epsilon, |b|\leq \epsilon\bigg\}.$$ 
For any given $g$, there exists (by the discreteness of $\Gamma$) some $\epsilon > 0$ so that $\cS \cap \gamma \cS \neq \emptyset$ for $\gamma \in \Gamma$ implies $g^{-1}  \gamma g \in n(\R)$.
  We note that such $\gamma$ exists only if $gn(\R)$ is closed.

  Let $\ell(g)$ be the length of the closed horocycle $g n(\xi)$ or $\ell(g) = \infty$ if $gn(\R)$ is not closed.
  Then, using Lemma \ref{lem:rhotransbounds} and the fact that in the coordinates $g_1 = g n(\xi) \begin{pmatrix}
    a^{-1} & 0 \\
    ab & a
  \end{pmatrix}$ Haar measure is $\dd\mu(g_1) = c_0\, a \dd a \dd b \dd \xi$ (for some $c_0>0$), we have
  \begin{multline*}
    \label{eq:thickenBound}
    \int_{-10}^{10} \alpha_1(g n(\xi) a(t))^\beta \dd \xi \ll (1 + \ell(g)^{-1}) \int_{\Gamma \backslash \Gamma \cS} \alpha_1( g_1 a(t))^\beta \dd \mu(g_1) \\
    \leq (1 + \ell(g)^{-1}) \int_{\Gamma \backslash G} \alpha_1( g_1)^\beta \dd \mu(g_1) \ll 1 + \ell(g)^{-1}.
  \end{multline*}

  For $d > 1$, we first note that since $\alpha_d \gg 1$, we may assume $\beta > \tfrac{1}{d}$.
  Now fix $s > 0$ large enough so that, applying Lemma \ref{lem:crudecontraction},
  \begin{equation}
    \label{eq:sfixeddef}
    \int_{-1}^1 \alpha_d(gn(\xi)a(s))^\beta \dd \xi \leq \tfrac{1}{2} \alpha_d(g)^\beta + O(1) \int_{-1}^1 \alpha_{d-1}(gn(\xi)a(s))^\beta \dd \xi. 
  \end{equation}
  We write $t = ks + t_1$ with integer $k\geq 0$ and $0\leq t_1 < s$.
  By Lemma \ref{lem:rhotransbounds} we have
  \begin{equation}
    \label{eq:ttoksbound}
    \int_{-10}^{10} \alpha_d(gn(\xi) a(t))^\beta \dd \xi \ll \int_{-10}^{10} \alpha_d(gn(\xi) a(ks))^\beta \dd \xi
  \end{equation}
  since $\e^{d\beta t_1} \ll 1$ as $s$ is fixed.
  Using the fact that
  \[ \int_{-1}^1 \mathbbm{1}_{[-10,10]} ( \xi + \epsilon \xi_1) \dd \xi_1 = \max(0, \min( \tfrac{\xi + 10}{\epsilon} + 1, 2, 1 - \tfrac{\xi - 10}{\epsilon})) \geq \mathbbm{1}_{[-10,10]}(\xi),\]
  we now show that the right side of \eqref{eq:ttoksbound} is bounded by iterating the following.
  \begin{equation}
  \begin{split}
    \label{eq:contractioniteration}
&    \int_{-10}^{10} \alpha_d(gn(\xi) a(ks))^\beta \dd \xi \\
 &   \leq  \int_{-10}^{10} \int_{-1}^1 \alpha_d(gn(\xi) a((k-1)s)n(\xi_1)a(s))^\beta \dd \xi_1 \dd \xi \\
  &  \leq \tfrac{1}{2} \int_{-10}^{10} \alpha_d(gn(\xi) a((k-1)s))^\beta \dd \xi + O(1) \int_{-10}^{10} \alpha_{d-1}(gn(\xi) a(ks))^\beta \dd \xi \\
  &  \leq \tfrac{1}{2} \int_{-10}^{10} \alpha_d(gn(\xi) a((k-1)s))^\beta \dd \xi + O( 1),
  \end{split}
  \end{equation}
  where the last bound follows by induction.
  The bound \eqref{eq:smallmoment} now follows from the case $k=0$, which is implied by Lemma \ref{lem:rhotransbounds}. 
  \end{proof}

\subsection{Proof of Theorem \ref{thm:escapeofmass}}
\label{sec:propproof}

In the following fix $g_0$ as in Theorem \ref{thm:escapeofmass}.
Let $K>1$ be large and $\delta=\frac{c}{K}$ for some constant $c>0$ to be determined later.
Let $T$ be a large constant to be determined.
Assuming $\delta < 1$, we let $t = \log M$ be large enough so that
\[ \frac{ \log ( 1 + \tfrac{\delta t}{T})}{\log (1 + \delta)} - \frac{ \log ( 1 + \tfrac{\delta t}{2T})}{\log (1 + \delta)} > 1\]
and we let $N$ be an integer in the interval $(\frac{ \log ( 1 + \tfrac{\delta t}{2T})}{\log (1 + \delta)}, \frac{\log ( 1 + \tfrac{\delta t}{T})}{\log (1 + \delta)})$.

Set $s_N = \frac{\delta \log M}{ (1 + \delta)^N -1}$ and note that $T < s_N < 2T$.
For $1 \leq i \leq N -1$, set $s_i = (1 + \delta)^{N-i} s_N$.
Then the sequence $\set{s_i}_{1\leq i\leq N}$ satisfies
        \begin{equation}
        \label{eq:sibound}
        \log M=s_1+\cdots+s_N
         \end{equation}
      and
    \begin{equation}
      \label{eq:sumratio}
      \left(1-\frac{s_N}{s_i}\right)\del^{-1}s_i= s_{i+1}+\cdots+s_N\leq \del^{-1}s_i 
      \textrm{ for any }1\leq i\leq N-1.
  \end{equation}
    
We recall from Lemma \ref{lem:crudecontraction} that
    \eqlabel{eq:recalltrivialcontraction}{
        \int_{-1}^{1}\alpha_3\big(hn(\xi)a(s)\big)^{\beta}\dd\xi\leq  e^{(3\beta-2 + \epsilon ) s}\alpha_3(h)^{\beta}+e^{O(1)s}\alpha_{2}(h)^{\beta}}
    for any $\Gamma^3 h\in \Gamma^3\backslash G^3$ and sufficiently large $s$, and from Proposition \ref{prop:globalcontraction}
    \begin{equation}\label{eq:recallnontrivialcontraction}
  \int_{-1}^{1}\widetilde{\alpha}(hn(\xi)a(s))\dd\xi
  \ll
  \e^{-3(2-\beta-\epsilon)s}\widetilde{\alpha}(h)
  +
  \e^{(K+O(1))s}
  \int_{-1}^{1}\alpha_2(hn(\xi)a(s))^\beta\dd\xi
\end{equation}
      for any $\Gamma^3 h\notin\cE_{s,K}$ and sufficiently large $s$.
      Here and in what follows implied constants $\ll$ may depend on $K$ and $\delta$, but $O(\cdot)$ in the exponent are independent of $K$ and $\delta$.

      We note that from Lemma \ref{lem:alpha3supbound} and Proposition \ref{prop:smallmoment},
\begin{multline}
  \label{eq:boundedxi}
  \tfrac{1}{L} \int_{\e^{-(1 + \eta)t} < |\xi| \leq 10} \alpha_3(g_0 n(\xi)a(t))^\beta \dd \xi \\
  \ll \tfrac{1}{L} \e^{(\beta - \frac{2}{3} + \epsilon_1) (\frac{3\kappa}{\kappa + 1} + 3\eta)t} \int_{|\xi| \leq 10} \alpha_3(g_0 n(\xi)a(t))^{\frac{2}{3} - \epsilon_1} \dd \xi \ll_\epsilon \tfrac{1}{L} \e^{((3\beta - 2)(\frac{\kappa}{\kappa + 1} + \eta) + \epsilon)t}
\end{multline}
for any $\epsilon > 0$.
This contributes the second term on the right of (\ref{eq:escapeofmass}).

We have
\eqlabel{eq:Ztexpression}{L^{-1}\int_{10}^L \alpha_3\big(g_0n(\xi)a(t)\big)^{\beta}\dd\xi\leq 2\mathsf{Z}_t,}
where
$$\mathsf{Z}_t:=L^{-1}\int_{3}^{3L}\int \alpha_3\big(g_0n(\xi_0)n(\xi_1)a(s_1)\cdots n(\xi_N)a(s_N)\big)^{\beta}\dd m_I^{\otimes N}(\xi_1,\cdots,\xi_{N})\dd\xi_0$$
with $I = [-1,1]$ and $\dd m_I(\xi)=\tfrac{1}{2}d\xi$ the normalized Lebesgue probability measure on $I$.
A similar bound holds for the integral over $-L \leq \xi \leq -10$, so we shall now prove that $\displaystyle\sup_{t>0}\mathsf{Z}_t<\infty$.
    
    For $1 \leq m \leq N-1$, let us define
    $$\overline{\Theta}_m:=\set{(\xi_0,\cdots,\xi_m)\in [3,3L]\times I^{m}: \Gamma^3 g_0n(\xi_0)n(\xi_1)a(s_1)\cdots n(\xi_m)a(s_m)\in \cE_{s_{m+1},K}},$$
    $$\Theta_m:=\overline{\Theta}_m\times I^{N-m}\subseteq [3,3L]\times I^{N},\ \Theta:=\bigcup_{m=1}^{N-1}\Theta_m.$$
    We have $\mathsf{Z}_t \leq \mathsf{Y}_t+\sum_{m=1}^{N-1}\mathsf{E}_{t,m}$, where
    $$\mathsf{Y}_t:=L^{-1}\int_{[3,3L]\times I^{N}\setminus\Theta} \alpha_3\big(g_0n(\xi_0)n(\xi_1)a(s_1)\cdots n(\xi_N)a(s_N)\big)^{\beta}\dd m_I^{\otimes N}(\xi_1,\cdots,\xi_{N}) \dd\xi_0,$$
    $$\mathsf{E}_{t,m}:=L^{-1}\int_{\Theta_m} \alpha_3\big(g_0n(\xi_0)n(\xi_1)a(s_1)\cdots n(\xi_N)a(s_N)\big)^{\beta}\dd m_I^{\otimes N}(\xi_1,\cdots,\xi_{N}) \dd\xi_0$$
    for $1\leq m\leq N-1$. 
    
    We shall first estimate $\mathsf{E}_{t,m}$. Write
    $$\mathsf{E}_{t,m}=L^{-1}\int_{\overline{\Theta}_m}J(\xi_0,\xi_1,\cdots,\xi_m)\dd m_I^{\otimes m}(\xi_1,\cdots,\xi_m)\dd\xi_0,$$
    where
    $$J(\xi_0,\xi_1,\cdots,\xi_m):=\int_{I^{N-m}}\alpha_3\big(g_0n(\xi_0)n(\xi_1)a(s_1)\cdots n(\xi_N)a(s_N)\big)^{\beta}\dd m_I^{\otimes (N-m)}(\xi_{m+1},\cdots,\xi_{N}).$$
    
    Since $s_i\geq T$ for any $1\leq i\leq N$, we may apply \eqref{eq:recalltrivialcontraction} for $J(\xi_1,\cdots,\xi_m)$ repeatedly with $s=s_N,\cdots,s_{m+1}$, and get
    \eqlabel{eq:Badestimate1}{\begin{aligned}
        J(\xi_0,\xi_1,\cdots,\xi_m)&\leq e^{(3\beta-2 + \epsilon)(s_{m+1}+\cdots+s_N)}\alpha_{3}\big(g_0n(\xi_0)n(\xi_1)a(s_1)\cdots n(\xi_m)a(s_m)\big)^{\beta}\\&\quad+\sum_{i=m+1}^{N}e^{O(s_i)}e^{(3\beta-2)(s_{i+1}+\cdots+s_{N})}\alpha_{2}\big(g_0n(\xi_0)n(\xi_1)a(s_1)\cdots n(\xi_m)a(s_m)\big)^{\beta}\\&\leq (N-m)e^{(3\beta-2)(s_{m+2}+\cdots+s_{N})+O(s_{m+1})}\alpha_{3}\big(g_0n(\xi_0)n(\xi_1)a(s_1)\cdots n(\xi_m)a(s_m)\big)^{\beta}\\&\ll e^{\{(3\beta-2)\delta^{-1}+O(1)\}s_{m+1}}\alpha_{3}\big(g_0n(\xi_0)n(\xi_1)a(s_1)\cdots n(\xi_m)a(s_m)\big)^{\beta},
    \end{aligned}}
  using \eqref{eq:sumratio} in the penultimate inequality.
    
    Let $\varphi:=\alpha_3^{\beta}\cdot\mathds{1}_{\cE_{s_{m+1},K}}$. Write \(s_m'=s_1+\cdots+s_m\) for $1\leq m \leq N$. Since
\(s_m'\geq s_1\), \(t<(\delta^{-1}+1)s_1\), and
\(L\leq\e^{(1+\theta)t}\), we have \(10L\leq\e^{Ds_m'}\) for sufficiently large \(t\), where \(D=(1+\theta)(\delta^{-1}+1)+2\). Thus Proposition \ref{prop:avoidanceMargulis} applies at every intermediate time \(s_m'\) with \(D=(1+\theta)(\delta^{-1}+1)+2\). By Proposition \ref{prop:avoidanceMargulis} we have
    \eqlabel{eq:thetabarint}{\begin{aligned}
        L^{-1}\int_{\widebar{\Theta}_m}\alpha_3&\big(g_0n(\xi_0)n(\xi_1)a(s_1)\cdots n(\xi_m)a(s_m)\big)^{\beta}\dd m_I^{\otimes m}(\xi_1,\cdots,\xi_m)\dd\xi_0\\&=L^{-1}\int_{[3,3L]\times I^{m}}\varphi\big(g_0n(\xi_0)n(\xi_1)a(s_1)\cdots n(\xi_m)a(s_m)\big)\dd m_I^{\otimes m}(\xi_1,\cdots,\xi_m)\dd\xi_0\\&\leq L^{-1}\int_{1}^{10L}\varphi(g_0n(\xi)a(s_1+\cdots+s_m))\dd\xi\ll e^{-(\mu K+O(1))s_{m+1}}.
    \end{aligned}}
In combination with \eqref{eq:Badestimate1} it follows that
    \eqlabel{eq:Badestimate3}{\begin{aligned}\mathsf{E}_{t,m}&\ll e^{\{(3\beta-2)\delta^{-1}+O(1)\}s_{m+1}}  \\ & \quad \times L^{-1} \int_{\widebar{\Theta}_m}\alpha_3\big(g_0n(\xi_0)n(\xi_1)a(s_1)\cdots n(\xi_m)a(s_m)\big)^{\beta}\dd m_I^{\otimes m}(\xi_1,\cdots,\xi_m) \dd\xi_0 \\&\ll e^{-\{\mu K-(3\beta-2)\delta^{-1}+O(1)\}s_{m+1}} 
    \end{aligned}}
    for any $1\leq m\leq N-1$. Therefore, if 
    \eqlabel{eq:Condition1}{c\mu=\mu \delta K >3\beta -2}
    then we have
    \eqlabel{eq:Ebound}{\begin{aligned}
        \sum_{m=1}^{N-1}\mathsf{E}_{t,m} &\ll  \sum_{m=1}^{N-1}e^{-\{\mu K-(3\beta-2)\delta^{-1}+O(1)\}s_{m+1}}\\&< \sum_{m=1}^{\infty} e^{-(1+\del)^m\{\mu K-(3\beta-2)\delta^{-1}+O(1)\}T}\\&<\sum_{m=1}^{\infty} e^{-m\del\{\mu K-(3\beta-2)\delta^{-1}+O(1)\} T}\\&=\sum_{m=1}^{\infty} e^{-m\{(c\mu -(3\beta-2))+O(K^{-1})\} T}\ll 1
    \end{aligned}}
    for sufficiently large $K$.

    We now estimate $\mathsf{Y}_t$ using \eqref{eq:recallnontrivialcontraction}. Observe that for any $1\leq m\leq N-1$ and $(\xi_0, \xi_1,\cdots,\xi_{N})\in [3,3L]\times I^{N}\setminus\Theta$ we have 
    $$\Gamma^3 g_0n(\xi_0)n(\xi_1)a(s_1)\cdots n(\xi_m)a(s_m)\notin \cE_{s_{m+1},K}$$ from the construction of $\Theta$.
    We also have the following.
   \begin{lem}
    \label{lem:initialDiophantine}
    For $3 \leq |\xi| \leq 3L$, $L \leq e^{(1 + \theta) t}$, $\delta$ sufficiently small and $K$ sufficiently large, we have for $1 \leq (1 + \theta)\ \kappa < \tfrac{1}{2} K \delta$ that $g_0 n(\xi) \not\in \cE_{s_1, K}$.
  \end{lem}

  \begin{proof}
    Let $\gamma = \gamma(g_0 n(\xi)n(\xi_1)a(s_1))= \left(
    \begin{pmatrix}
      * & * \\
      c_j & d_j 
    \end{pmatrix}
    \right)_{\!\! j=1,2,3}$.
    We may assume that $c_j \neq 0$, then $|\xi + \xi_1 + \frac{d_j}{c_jx_j}| \ll |c_j|^{-1} \e^{-s_1}$, so in particular $|\xi + \xi_1| \geq 2$ implies that $d_j \neq 0$.
    Now by the Diophantine condition on $\tfrac{x_i}{x_j}$, for $c_ic_j \neq 0$, we have
    \begin{equation}
      \label{eq:initialexceptional}
      |\tfrac{d_i}{c_i x_i } - \tfrac{d_j}{c_j x_j}| = |c_ic_j|^{-1} |c_j d_i \tfrac{x_j}{x_i} - c_i d_j| \gg |c_i c_j|^{-1}|c_j d_i|^{-\kappa} \gg L^{-\kappa} \e^{-(2\kappa + 2)s_1}
    \end{equation}
    from $|d_j| \ll L |c_j|$ and $|c_j| \ll e^{s_1}$. 

    Since $t=\sum_{i=1}^Ns_i
  =\delta^{-1}\big((1+\delta)s_1-s_N\big)
  <(\delta^{-1}+1)s_1$, we have
\[
  L^{-\kappa}\e^{-(2\kappa+2)s_1}
  \geq
  \exp\left(
    -\{(1 +\theta) \kappa(\delta^{-1}+1)+2\kappa+2\}s_1
  \right).
\]
Writing $c=K\delta$, the assumption $c>2(1 + \theta) \kappa$ implies that the
last expression is greater than $\e^{-Ks_1/2}$ for sufficiently
large $K$.
  \end{proof}

    Applying \eqref{eq:recallnontrivialcontraction} with $s=s_N,\cdots,s_1$ repeatedly, we get
    \eqlabel{eq:Goodestimate1}{\begin{aligned}
        \mathsf{Y}_t&<L^{-1}\int_{[3,3L]\times I^{N}\setminus\Theta}\alpha_3\big(g_0n(\xi_0)n(\xi_1)a(s_1)\cdots n(\xi_N)a(s_N)\big)^{\beta}\dd m_I^{\otimes N}(\xi_1,\cdots,\xi_{N})\dd\xi_0\\&\leq L^{-1}\int_{[3,3L]\times I^{N}\setminus\Theta} \widetilde{\alpha}\big(g_0n(\xi_0)n(\xi_1)a(s_1)\cdots n(\xi_N)a(s_N)\big)\dd m_I^{\otimes N}(\xi_1,\cdots,\xi_{N})\dd\xi_0\\&\leq e^{-3(2-\beta - \epsilon) t}L^{-1}\int_{3}^{3L}\widetilde{\alpha}(g_0n(\xi_0))\dd\xi_0        \\&\qquad+\sum_{m=1}^{N}e^{(K+O(1))s_m-3(2-\beta)(s_{m+1}+\cdots+s_{N})}\\&\qquad\qquad\times L^{-1}\int_{3}^{3L}\int_{I^{m}}\alpha_{2}\big(g_0n(\xi_0)n(\xi_1)a(s_1)\cdots n(\xi_m)a(s_m)\big)^{\beta}\dd m_I^{\otimes m}(\xi_1,\cdots,\xi_{m})\dd\xi_0\\&\ll e^{-3(2-\beta-\epsilon) t}L^{-1}\int_{3}^{3L}\widetilde{\alpha}(g_0n(\xi_0))\dd\xi_0 +\sum_{m=1}^{N}e^{(K+O(1))s_m-3(2-\beta)(s_{m+1}+\cdots+s_{N})}.
    \end{aligned}}
    We choose a sufficiently small $\epsilon>0$ and let $N'$ be the largest integer such that $s_{N'}\geq \epsilon^{-1}s_N$. Note that $s_{N'} \leq \epsilon^{-1} (1 + \delta) s_N$, so $N-N'\ll \delta^{-1}\log\frac{1}{\eps}$, since $s_{N'}=(1+\delta)^{N-N'}s_N$. If $1\leq m\leq N'$, then $$s_{m+1}+\cdots+s_{N}\geq \left(1-\frac{s_N}{s_{N'}}\right)\delta^{-1}s_m\geq(1-\epsilon)\delta^{-1}s_m.$$
    Hence, if 
    \eqlabel{eq:Condition2}{c=K\delta<3(1-\epsilon)(2-\beta)} then we have
    \eqlabel{eq:exponentialsum0}{\begin{aligned}
        \sum_{m=1}^{N'}e^{(K+O(1))s_m-3(2-\beta)(s_{m+1}+\cdots+s_{N})}&\leq \sum_{m=1}^{N'}e^{\{(K+O(1))-3(1-\epsilon)(2-\beta)\delta^{-1}\}s_m}\\&\ll\sum_{m=1}^{\infty}e^{\{(K+O(1))-3(1-\epsilon)(2-\beta)\delta^{-1}\}(1+\delta)^mT}\\&\leq \sum_{m=1}^{\infty}e^{\{(c-3(1-\epsilon)(2-\beta))+O(K^{-1)}\}m T}<\infty
    \end{aligned}}
    uniformly for sufficiently large $K$.

    Now we fix $K$ sufficiently large so that
Lemma \ref{lem:initialDiophantine}, \eqref{eq:Ebound} and \eqref{eq:exponentialsum0} hold.
Since $\mu>\tfrac13$, $\kappa$ is sufficiently close to $1$ and $\theta < \tfrac{3}{2\kappa}-1$,
we may choose $\beta>1$ sufficiently close to $1$ and
$\epsilon>0$ sufficiently small so that
\[
  \max\left\{
    \mu^{-1}(3\beta-2),\,2(1+\theta) \kappa
  \right\}
  <
  3(1-\epsilon)(2-\beta).
\]
We then choose $c$ such that
\[
  \max\left\{
    \mu^{-1}(3\beta-2),\,2(1+\theta) \kappa
  \right\}
  <c<
  3(1-\epsilon)(2-\beta)
\]
and set $\delta=\frac{c}{K}$. This gives both
\eqref{eq:Condition1} and \eqref{eq:Condition2}, as well as the
assumption $(1+\theta) \kappa<\tfrac12K\delta=\tfrac c2$ in Lemma
\ref{lem:initialDiophantine}.
    
    It is also easy to see
    \eqlabel{eq:exponentialsum2}{\begin{aligned}
        \displaystyle\sum_{m=N'+1}^{N}e^{(K+O(1))s_m-3(2-\beta)(s_{m+1}+\cdots+s_{N})}&\leq (N-N')e^{(2K+O(1))s_N}\\& \ll\delta^{-1}\log\frac{1}{\epsilon}e^{(2K+O(1))s_N}<\infty.
    \end{aligned}}
  
  Combining \eqref{eq:Goodestimate1}, \eqref{eq:exponentialsum0}, and \eqref{eq:exponentialsum2} with Lemma \ref{lem:initialtildealphaaverage} below, $\mathsf{Y}_t$ is uniformly bounded for any large $t$.
  In combination with \eqref{eq:Ebound}, it follows that
    $$\mathsf{Z}_t\leq\mathsf{Y}_t+\sum_{m=1}^{N-1}\mathsf{E}_{t,m}$$
    is also uniformly bounded for any large $t$, hence $\displaystyle\sup_{t>0}\mathsf{Z}_t<\infty$.
    This completes the proof of Theorem \ref{thm:escapeofmass} once we have the following.

  \begin{lem}
    \label{lem:initialtildealphaaverage}
    We have
    \begin{equation}
      \label{eq:initialtildealphaaverage}
      \tfrac{1}{L} \int_3^{3L} \widetilde{\alpha}(g_0 n(\xi)) \dd \xi \ll 1. 
    \end{equation}
  \end{lem}

  \begin{proof}
    From the scaling \eqref{eq:xnormalization} we have $\gamma_j(g_0 n(\xi)) =
    \begin{pmatrix}
      1 & * \\
      0 & 1
    \end{pmatrix}
    $, whence $\Delta_3(\gamma(g_0 n(\xi)) g_0 n(\xi)) = 1$ and $\tilde \alpha (g_0 n(\xi)) \ll 1$.
    The bound \eqref{eq:initialtildealphaaverage} follows.    
    \end{proof}

\section{Avoidance estimate}
\label{sec:avoidance}

In this section we prove the following avoidance estimate, which is used in the proof of Theorem \ref{thm:escapeofmass}, see \eqref{eq:thetabarint}.
\begin{prop}\label{prop:avoidanceMargulis}
  There exist $\kappa_0 > 1$ and $\mu > \tfrac{1}{3}$ such that for $1 \leq \kappa < \kappa_0$ there is $\beta = \beta(\kappa) > 1$ with the following property. Let $g_0 = \left(\begin{pmatrix}    x_j^{\frac{1}{2}} & 0 \\
    0 & x_j^{-\frac{1}{2}}
  \end{pmatrix}
  \right)_{\!\! j=1,2,3}$ with $\tfrac{x_i}{x_j}$ $\kappa$-Diophantine, $\min x_j \geq 1$, and let $D\geq1$. Then for any $1\leq L \leq e^{Dt}$ and $K, s \geq 0$ we have
  \begin{equation}
    \label{eq:avoidanceMargulis}
    \int_{1\leq |\xi| \leq  L} \alpha_3\big(g_0n(\xi)a(t)\big)^{\beta}\mathds{1}_{\cE_{s,K}}(\Gamma^3 g_0n(\xi)a(t))\dd\xi\ll_{D,K} e^{-(\mu K+O(1))s}L.
  \end{equation}
  The implied constant may depend on \(D\) and \(K\), while the
\(O(1)\) in the exponent is independent of both \(D\) and \(K\). 
\end{prop}

\subsection{Avoidance in $\Gamma^2 \backslash G^2$}
\label{sec:rank2avoidance}

Towards proving Proposition \ref{prop:avoidanceMargulis}, we establish the following in $\Gamma^2 \backslash G^2$.
For $g =\left(
\begin{pmatrix}
  * & * \\
  c_j & d_j 
\end{pmatrix}
\right)_{\!\! j=1,2}
\in G^2$, we let $\Delta_2(g) = \min(1, |\tfrac{d_1}{c_1} - \tfrac{d_2}{c_2}|)$ and we set
\begin{multline}
  \label{eq:cE'def}
  \cE'_{s,K} = \{ \Gamma^2 g \in \Gamma^2 \backslash G^2: \Delta_2(\gamma g) \leq \e^{-Ks}\textrm{ for some } \gamma \in \Gamma^2\\
  \textrm{ such that } \max_{j=1,2} \rho(\gamma_j g_j) < 10 \e^s \}. 
\end{multline}
We note the similarity between $\cE'_{s,K}$ and $\tilde\cE_{s,K}$ in (\ref{eq:tildecEdef}) apart from the lack of the lower bound for $\rho$ in (\ref{eq:cE'def}).
Roughly speaking, $\cE'_{s,K}$ is obtained from $\tilde\cE_{s,K}$ by removing the coordinate $j$ such that $\rho(\gamma_j g_j) > \tfrac{1}{10} \e^{-s}$. 
\begin{prop}
  \label{prop:rank2avoidance}
  Let $\kappa\geq 1$, $g_0 = \left(\begin{pmatrix}    x_j^{\frac{1}{2}} & 0 \\
    0 & x_j^{-\frac{1}{2}}
  \end{pmatrix}
  \right)_{\!\! j=1,2}$ with $\tfrac{x_1}{x_2}$ $\kappa$-Diophantine, $\min x_j \geq 1$, and $D\geq 1$.
  Then for any $\beta < 1 + \kappa^{-2}$, $1\leq L \leq \e^{Dt}$, and $K,s\geq 0$ we have
  \begin{equation}
    \label{eq:rank2moment1}
    \int_{1 \leq |\xi| \leq L} \alpha_2(g_0 n(\xi) a(t))^\beta \mathbbm{1}_{\cE'_{s,K}}(g_0n(\xi)a(t)) \dd \xi \ll_{D,K} L \e^{-(\kappa^{-2} K + O(1))s} .
  \end{equation}
  The implied constant may depend on \(D\) and \(K\), while the \(O(1)\) in the exponent is independent of both \(D\) and \(K\).
\end{prop}

The proof of Proposition \ref{prop:rank2avoidance} requires the following lemmas.
For $1\leq C_j \ll \e^t$, define
\begin{multline}
  \label{eq:cCdef2}
  \cC_{t,\eta, L}(C_1, C_2) = \{(c_j, d_j)_{j=1,2} : \gcd(c_j,d_j) = 1,\ C_j \leq |c_j| < 2 C_j,\\
  0 < |d_j| \leq 2x_j LC_j,\ |\tfrac{d_1}{c_1x_1} - \tfrac{d_2}{c_2x_2}| \leq \eta \e^{-2t}\}. 
\end{multline}

\begin{lem}
  \label{lem:cCcount}
  Suppose that $\tfrac{x_1}{x_2}$ is $\kappa$-Diophantine.
  \begin{enumerate}[{\rm (i)}]
  \item We have $|\cC_{t, \eta, L}(C_1, C_2)| = 0$ unless $(L C_1C_2)^{-\kappa} \ll \eta C_1C_2 \e^{-2t}$. 
    
  \item
    We have
    \begin{equation}
      \label{eq:cCcount}
      |\cC_{t, \eta, L}(C_1,C_2)|\ll_\epsilon
      \begin{cases}
        (\eta C_1C_2\e^{-2t} +(LC_1C_2)^{-\frac{1}{\kappa}})(LC_1C_2)^{1+\epsilon} \\
        (\eta C_1C_2\e^{-2t} +(\eta C_1C_2\e^{-2t})^{\kappa^{-2}})(LC_1C_2)^{1+\epsilon}
      \end{cases}
    \end{equation}
    for any $\epsilon > 0$.
  \end{enumerate}
\end{lem}

\begin{proof}
  We first note that since $c_1d_2 \neq 0$,
  \begin{equation}
    \label{eq:Diophantinelowerbound}
    |c_1d_2 \tfrac{x_1}{x_2} - c_2d_1| \gg |c_1d_2|^{-\kappa} \gg (LC_1C_2)^{-\kappa}.
  \end{equation}
  and (i) follows.
  
  For (ii), let $c = c_1d_2$ and $d = c_2d_1$.
  Let $1\leq q \leq LC_1C_2$ be so that $|\tfrac{x_1}{x_2} - \tfrac{r}{q}| \leq \tfrac{1}{LC_1C_2 q}$ for some integer $r$ coprime to $q$.
  Since $\frac{x_1}{x_2}$ is $\kappa$-Diophantine, we have $q^{\kappa}\gg LC_1C_2$.
  
  From $|c\tfrac{x_1}{x_2} - d| \ll \eta C_1 C_2 \e^{-2t}$, we have
  \begin{equation}
    \label{eq:hdef}
    h := cr - dq \ll \eta q C_1C_2 \e^{-2t}  + 1. 
  \end{equation}
    By assertion (i), we may assume $$\eta C_1C_2 \e^{-2t} \gg (LC_1C_2)^{-\kappa}\gg q^{-\kappa^2},$$
    since otherwise $\cC_{t, \eta, L}(C_1, C_2)$ is empty. 
  
  We note that $c \equiv h \overline{r} \bmod q$, and $d = \tfrac{cr - h}{q}$, so
  \begin{equation}
  \begin{split}
    \label{eq:cCcount1}
    |\cC_{t, \eta, L}(C_1, C_2)| & \ll \sum_{|h| \ll \eta q C_1 C_2 \e^{-2t} +1} \sum_{\substack{|c| \leq LC_1C_2 \\ c \equiv h\overline{r} \bmod q}} \tau(c) \tau(\tfrac{cr - h}{q})\\ & \ll \tfrac{1}{q} (\eta qC_1C_2 \e^{-2t} +1)(LC_1C_2)^{1+\epsilon},
  \end{split}
  \end{equation}
  where $\tau$ is the divisor function.
\end{proof}

Combined with Lemma \ref{lem:cCcount}, \cite[Theorem~1]{LMW2023a} gives the following.
\begin{lem}
  \label{lem:rank2moments1}
  Under the assumptions of Proposition \ref{prop:rank2avoidance}, for any $1\leq L \leq \e^{Dt}$ and $\beta < 1 + \kappa^{-2}$, we have
  \begin{equation}
    \label{eq:rank2moment1B}
    \int_{1\leq |\xi| \leq L} \alpha_2(g_0 n(\xi) a(t))^\beta \dd \xi \ll_D L.
  \end{equation}
\end{lem}

\begin{proof}
  We divide the integral in \eqref{eq:rank2moment1B} as
  \begin{equation}
    \label{eq:largeSmallSplit}
    \int_{\substack{1 \leq |\xi| \leq L \\ \alpha_2(g_0n(\xi)a(t)) \leq \e^{\epsilon_1 t}}} \alpha_2(g_0 n(\xi) a(t))^\beta \dd \xi + \int_{\substack{ 1\leq |\xi| \leq L \\ \alpha_2(g_0n(\xi)a(t)) > \e^{\epsilon_1 t}}} \alpha_2(g_0 n(\xi) a(t))^\beta \dd \xi. 
  \end{equation}
  Since the indicator function of the region where $\alpha_2(g) \leq \e^{\epsilon_1 t}$ has a smooth compactly supported majorant with Sobolev norm $\ll \e^{O(\epsilon_1 t)}$, the first integral in \eqref{eq:largeSmallSplit} is bounded by means of Proposition \ref{prop:effectiveequidistribution}, noting that for $\beta < 2$,
  \begin{equation}
    \label{eq:alpha2average}
    \int_{\Gamma^2 \backslash G^2} \alpha_2(g)^\beta \dd \mu(g) \ll 1
  \end{equation}
  and that by taking $\epsilon_1$ sufficiently small, the remainder term on the right of \eqref{eq:effectiveequidistribution} is also bounded. 

  For the second integral in \eqref{eq:largeSmallSplit}, we dyadically decompose according to $\tfrac{1}{2} \delta_j^{-1} < \alpha_1(g_j n(\xi) a(t))\leq \delta_j^{-1}$. 
  We note that from Lemma \ref{lem:basicfundamentaldomain}, $\gamma(g_0n(\xi)a(t)) = \left(\begin{pmatrix}    * & * \\
    c_j & d_j
  \end{pmatrix}
  \right)_{\!\! j=1,2}$ satisfies $0 < |c_j| \ll \delta_j \e^t$, $d_j \neq 0$, and $|\xi + \tfrac{d_j}{c_jx_j}| \ll \tfrac{\delta_j}{|c_j|\e^t}$, so in particular $|\tfrac{d_1}{c_1 x_1} - \tfrac{d_2}{c_2 x_2}| \ll (\tfrac{\delta_1}{|c_1|} + \tfrac{\delta_2}{|c_2|}) \e^{-t}.$
  Moreover, given such a $\gamma$, the measure of the associated $\xi$ is $O ( \min ( \tfrac{\delta_j}{|c_j|\e^t}))$ and therefore we obtain the bound
  \begin{multline}
    \e^{-t}\sum_{\substack{\delta_1\delta_2 \ll \e^{-\epsilon_1 t} \\ \delta_j \ll 1\\ \mathrm{dyadic}}} (\delta_1\delta_2)^{-\beta} \sum_{\substack{1\leq C_j \ll \delta_j \e^t \\ \mathrm{dyadic}}} \min( \tfrac{\delta_j}{C_j}) |\cC_{t,\e^t\max(\frac{\delta_j}{C_j}),L}(C_1, C_2)|.
  \end{multline}

  We bound $\cC$ using Lemma \ref{lem:cCcount}.
  For $\epsilon_2 > 0$, the first term on the right of \eqref{eq:cCcount} contributes
  \begin{multline}
    \label{eq:cCcountfirstterm}
    L^{1+\epsilon_2} \e^{-2t} \sum_{\substack{\delta_1\delta_2 \ll \e^{-\epsilon_1 t}\\ \delta_j \ll 1 \\ \mathrm{dyadic}}} (\delta_1\delta_2)^{1 -\beta} \sum_{\substack{1\leq C_j \ll \delta_j \e^t \\ \mathrm{dyadic}}} (C_1C_2)^{1 + \epsilon_2} \\
    \ll L^{1 + \epsilon_2} e^{2\epsilon_2 t} \sum_{\substack{\delta_1\delta_2 \ll \e^{-\epsilon_1 t}\\ \delta_j \ll 1 \\ \mathrm{dyadic}}} (\delta_1\delta_2)^{2 -\beta +\epsilon_2} \ll L^{1 + \epsilon_2} \e^{2\epsilon_2 t} \e^{-(2 -\beta + \epsilon_2) \epsilon_1 t}.
  \end{multline}
  Since $L\leq\e^{Dt}$,
  we can choose $\epsilon_2$ sufficiently small depending on $D$ and $\epsilon_1$ so that this factor is absorbed by the exponential saving in $t$.

  The second term in the last line on the right of \eqref{eq:cCcount} contributes 
  \begin{multline}
    \label{eq:cCcountsecondterm}
    L^{1 + \epsilon_2 } \e^{-(1 + \kappa^{-2})t} \sum_{\substack{ \delta_1\delta_2 \ll \e^{-\epsilon_1 t} \\ \delta_j \ll 1\\ \mathrm{dyadic}}} (\delta_1 \delta_2)^{- \beta}
    \sum_{\substack{1\leq C_j \ll \delta_j \e^t \\\mathrm{dyadic}}} \min(\frac{\delta_j}{C_j}) \max( \frac{\delta_j}{C_j})^{\kappa^{-2}}(C_1C_2)^{1 + \kappa^{-2} + \epsilon_2} \\
    \leq 2 L^{1 + \epsilon_2 } \e^{-(1 + \kappa^{-2})t} \sum_{\substack{ \delta_1\delta_2 \ll \e^{-\epsilon_1 t} \\ \delta_j \ll 1\\ \mathrm{dyadic}}} \delta_1^{1-\beta} \delta_2^{\kappa^{-2} - \beta}
    \sum_{\substack{1\leq C_1 \ll \delta_1 \e^t \\ 1\leq C_2 \leq \delta_1^{-1} \delta_2 C_1 \\\mathrm{dyadic}}} C_1^{\kappa^{-2} + \epsilon_2}C_2^{1 + \epsilon_2} \\
    \ll L^{1 + \epsilon_2 } \e^{-(1 + \kappa^{-2})t} \sum_{\substack{ \delta_1\delta_2 \ll \e^{-\epsilon_1 t} \\ \delta_j \ll 1\\ \mathrm{dyadic}}} \delta_1^{-\beta - \epsilon_2}\delta_2^{1 + \kappa^{-2} - \beta + \epsilon_2} 
    \sum_{\substack{1\leq C_1 \ll \delta_1 \e^t  \\\mathrm{dyadic}}} C_1^{1 + \kappa^{-2} + 2\epsilon_2}\\
    \ll L^{1 + \epsilon_2}\e^{2\epsilon_2 t} \sum_{\substack{ \delta_1\delta_2 \ll \e^{-\epsilon_1 t} \\ \delta_j \ll 1\\ \mathrm{dyadic}}} (\delta_1 \delta_2)^{1 + \kappa^{-2}- \beta + \epsilon_2} \ll  L^{1 + \epsilon_2}\e^{2\epsilon_2 t} \e^{-(1 + \kappa^{-2} - \beta + \epsilon_2)\epsilon_1 t}
  \end{multline}
  since $\beta < 1 + \tfrac{1}{\kappa^2}$.
  As above, this is $O(L)$ by making $\epsilon_2$ sufficiently small.  
\end{proof}

\begin{proof}[Proof of Proposition \ref{prop:rank2avoidance}]
  Let us first consider the case when $s t^{-1}$ is small.
By Lemma \ref{lem:rank2moments1}, the set $\cS$ of $\alpha_2(g) > \e^{\epsilon_1t}$ contributes $O(e^{-\epsilon_1 \beta' t})$ with $\beta' = \tfrac{1}{2}( 1+ \kappa^{-2} - \beta)$.
Now for $g = \left(
\begin{pmatrix}
  * & * \\
  0 & * 
\end{pmatrix}
\begin{pmatrix}
  \cos \theta_j & -\sin \theta_j \\
  \sin \theta_j & \cos \theta_j
\end{pmatrix}
\right)_{\!\! j=1,2}$, $\Delta_2(g)\leq \e^{-Ks}$ implies $|\theta_1 - \theta_2| \ll \e^{-Ks}$.
Moreover, for any $g = (g_1,g_2) \in G^2$, the number of $\gamma = (\gamma_1, \gamma_2) \in \Gamma^2$ modulo $\Gamma_\infty^2$ such that $\rho(\gamma_jg_j) \leq 10 \e^s$ is $\ll \e^{O(s)}$ \cite[Lemma~2.10]{Iwaniec2002}.
Therefore, we may find smooth majorants $\varphi_1$ and $\varphi_2$ for the indicator functions of $G \setminus \cS$ and $\cE'_{s,K}$ so that their Sobolev norms are $\ll \e^{O(\epsilon_1 t)}$ and $\ll \e^{O(K s)}$ respectively and
\begin{equation}
  \label{eq:varphijint}
  \int_{\Gamma^2\backslash G^2} \varphi_1(g)\varphi_2(g) \dd \mu(g) \ll \e^{-(K + O(1))s}.
\end{equation}
By Proposition \ref{prop:effectiveequidistribution}, the left side of \eqref{eq:rank2moment1} is bounded by
\begin{equation*}
  \ll \e^{(-\kappa^{-2} K + O(1)) s} + \e^{-\delta t + O(\epsilon_1 t + Ks)} + \e^{-\epsilon_1 \beta' t}. 
\end{equation*}
If $\epsilon_1$ is sufficiently small and $s t^{-1} \leq \delta(CK)^{-1}$ for a large constant $C$, then \eqref{eq:rank2moment1} follows. 

For the complementary range of $s t^{-1}$, we dyadically decompose as in the proof of Lemma \ref{lem:rank2moments1}, obtaining the bound
\begin{equation}
  \label{eq:avoidancedyadicdecomp}
  \ll \e^{-t}\sum_{\substack{\delta_j \ll \e^s\\ \mathrm{dyadic}}}(\delta_1 \delta_2)^{-\beta} \sum_{\substack{1\leq C_j \ll \delta_j \e^t \\ \mathrm{dyadic}}} \min(\tfrac{\delta_j}{C_j}) |\cC_{t,\e^{-Ks},L}(C_1,C_2)|.
\end{equation}
By Lemma \ref{lem:cCcount} this is bounded by
\begin{multline}
  \label{eq:avoidancedyadicdecomp1}
  \ll \e^{-\kappa^{-2} K s} L^{1 + \epsilon_2} \e^{-(1 - 2\kappa^{-2})t}\sum_{\substack{\delta_j \ll \e^s\\ \mathrm{dyadic}}}(\delta_1 \delta_2)^{\frac{1}{2} -\beta} \sum_{\substack{1\leq C_j \ll \delta_j \e^t \\ \mathrm{dyadic}}} (C_1C_2)^{\frac{1}{2} + \kappa^{-2} + \epsilon_2} \\
  \ll \e^{-\kappa^{-2}Ks} L^{1 + \epsilon_2} \e^{2\epsilon_2 t} \sum_{\substack{\delta_j \ll \e^s\\ \mathrm{dyadic}}}(\delta_1 \delta_2)^{1 + \kappa^{-2}-\beta} \ll \e^{(-\kappa^{-2} K + O(1))s} L \e^{\epsilon_2( D + 2) t}.
\end{multline}
Since \(t<CKs/\delta\) in this range, we choose
$\epsilon_2\leq c\delta\big(C(D+2)K\big)^{-1}$, for a sufficiently small absolute constant $c>0$. Then $\e^{\epsilon_2(D+2)t}\leq\e^{cs}$,
so the \(O(1)\) in the exponent is independent of \(D\) and \(K\), while the implied multiplicative constant may depend on \(D\) and \(K\).
\end{proof}

\subsection{Avoidance in $\Gamma^3 \backslash G^3$}
\label{sec:avoidanceproof}

We first decompose the exceptional set $\cE_{s,K}$ as follows:
\begin{equation}
  \label{eq:cEdecomp}
  \cE_{s,K}\subseteq\bigcup_{\substack{1\leq i\neq j \leq 3 \\ 1\leq k \leq 3}}\cE_{s,K}(i,j, k),
\end{equation}
where
\begin{multline}
  \label{eq:decompparts}
  \cE_{s,K}(i, j,k):=\{\Gamma^3 g\in \Gamma^3 \backslash G^3: 
  \Delta_{ij}(\gamma g)^2 \leq\Delta_{j}(\gamma g)< \e^{-Ks} \\
  \textrm{ for some } \gamma \in \Gamma^3 \textrm{ such that } \tfrac{1}{10} \e^{-s}\leq \rho(\gamma_k g_k),\ \max_{r=1,2,3} \rho(\gamma_r g_r) \leq 10 \e^{s}.\}
  \end{multline}
with
\begin{equation}
\label{eq:Deltaijdef}
\Delta_{ij}(g)=
\begin{cases}
\displaystyle
\min\left(1,
\left|\frac{d_i}{c_i}-\frac{d_j}{c_j}\right|\right),
&\textrm{ if }c_ic_j\neq0,\\[6pt]
1,&\textrm{ if } c_ic_j=0,
\end{cases}
\end{equation}
and
\begin{equation}
\label{eq:Deltadef2}
\Delta_j(g)=\prod_{i\neq j}\Delta_{ij}(g).
\end{equation}

Without loss of generality, we may assume $i=2$ and $j=3$ and estimate the contribution of $\cE_{s,K}(2,3,k)$.

When $k = 1$, we have that
\begin{multline}
  \label{eq:rank2reduction1}
  \int_{1\leq |\xi| \leq L} \alpha_3(g_0n(\xi) a(t))^\beta \mathbbm{1}_{\cE_{s,K}(2,3,1)}(g_0n(\xi)a(t)) \dd \xi \\
  \ll \e^{O(s)} \int_{1\leq |\xi| \leq L} \alpha_2(\pi_1(g_0)n(\xi) a(t))^\beta \mathbbm{1}_{\cE'_{s,K/2}}(\pi_1(g_0)n(\xi)a(t)) \dd \xi,
\end{multline}
where $\pi_1: \Gamma^3 \backslash G^3 \to \Gamma^2 \backslash G^2$ is given by $\pi_1(g_1,g_2, g_3)= (g_2, g_3)$.
Now \eqref{eq:avoidanceMargulis} with $\cE_{s,K}$ replaced with $\cE_{s,K}(2,3,1)$ follows from Proposition \ref{prop:rank2avoidance} with $\mu < \tfrac{1}{2\kappa^2}$. 

When $k=2$ (the case $k=3$ is similar), we first note that the locus $\Delta_{i3}=0$ with $i=1,2$ does not occur on the orbit segment under
consideration. Indeed, writing $(c_r,d_r)$ for the bottom row of
$\gamma_r$, the equality $\Delta_{i3}\bigl(\gamma g_0n(\xi)a(t)\bigr)=0$ implies $\frac{d_i}{c_i x_i}=\frac{d_3}{c_3x_3}$. If $d_id_3\neq0$, this contradicts the Diophantine condition on $x_i/x_3$, while the remaining case $d_i=d_3=0$ corresponds to the major arc $|\xi|<1$, which has already been excluded. Hence, for $1\leq|\xi|\leq L$, both gaps are positive and may be decomposed dyadically. Thus we have
$$\cE_{s,K}(2,3,2)\cap \mathcal{O}_{t,L} \subset \bigcup_{\substack{K_1 + K_2 \geq K-1 \\ K_2 \geq K_1 -1}} \cE_{s,K_1,K_2}(2,3,2)\cap \mathcal{O}_{t,L},$$
where $K_1,K_2$ are integers, $\mathcal{O}_{t,L}:=\{g_0n(\xi)a(t): 1\leq |\xi|\leq L\}$, and
\begin{multline}
  \label{eq:etadef}
  \cE_{s,K_1,K_2}(2,3,2) = \{ g \in \Gamma^3 \backslash G^3 : \e^{-(K_i +1)s} < \Delta_{i3}(\gamma g) \leq \e^{-K_i s} \; (i=1,2) \\
  \textrm{ for some } \gamma \in \Gamma^3 \textrm{ such that } \tfrac{1}{10} \e^{-s}\leq  \rho(\gamma_2 g_2),\ \max_{j=1,2,3} \rho(\gamma_j g_j) \leq 10 \e^{s}\}.
\end{multline}
Let us consider $K_2 \leq AK$ for some $A > 0$.
Let $1<p < \tfrac{1 + \kappa^{-2}}{\beta} < 2$.
By H\"older's inequality, we have
\begin{multline}
  \label{eq:holder1}
  \sum_{\substack{K_1 + K_2 \geq K-1 \\ K_2 \leq AK}} \int_{1 \leq |\xi| \leq L} \alpha_2(\pi_2(g_0)n(\xi) a(t))^\beta \chi_{K_1}(g_0n(\xi)a(t)) \chi_{K_2}(g_0n(\xi)a(t)) \dd \xi \\
  \leq \sum_{\substack{K_1 + K_2 \geq K-1 \\ K_2 \leq AK}}\left( \int_{1\leq |\xi| \leq L} \alpha_2(\pi_2(g_0)n(\xi) a(t))^{p \beta} \tilde{\chi}_{K_1} (\pi_2(g_0)n(\xi) a(t))\dd \xi\right)^{\tfrac{1}{p}} \\
  \times \left(\int_{1 \leq |\xi| \leq L} \tilde{\chi}_{K_2}(\pi_1(g_0)n(\xi) a(t)) \dd\xi \right)^{1 - \frac{1}{p}} ,
\end{multline}
with $\chi_{K_i}$ the indicator functions of the sets
\begin{multline}
  \label{eq:etasets}
  \{\Gamma^3 g\in \Gamma^3 \backslash G^3: \e^{-(K_i + 1)s} < \Delta_{i,3}(\gamma g) \leq \e^{-K_i s} \\
  \textrm{ for some } \gamma \in \Gamma^3 \textrm{ such that }
  \max_{1\leq k\leq 3}\rho(\gamma_k g_k)\leq 10 \e^s\}
\end{multline}
and $\tilde \chi_{K_i}$ is the indicator function of the associated projections of the sets \eqref{eq:etasets} to $\Gamma^2 \backslash G^2$ under $\pi_i: (g_1,g_2,g_3) \mapsto (g_i,g_3)$.
By Proposition \ref{prop:rank2avoidance}, \eqref{eq:holder1} is bounded by
\begin{equation}
  \label{eq:dyadicetasum}
  L \e^{O(s)}\sum_{\substack{K_1 + K_2 \geq K-1 \\ K_2 \leq AK}} \e^{-\frac{1}{\kappa^2 p}K_1 s} \e^{-\frac{1}{\kappa^2}(1 - \tfrac{1}{p})K_2 s }
  \ll L\e^{-((\frac{1}{\kappa^2 p} - \frac{A}{\kappa^2} (\frac{2}{p} - 1) )K + O(1))s}.
\end{equation}

Now consider the case $K_2 \geq AK$.
We let $\hat\pi_j : \Gamma^3 \backslash G^3 \to \Gamma\backslash G $ be given by $\hat\pi_j(g_1,g_2,g_3) = g_j$.
We let $\delta_1 > 0$ be so that $2\beta - \delta_1 < 1 + \tfrac{1}{\kappa^2}$, and we let $\tfrac{1}{p_1} + \tfrac{1}{q_1} = 1$ be so that $\beta_1 = 2p_1(\beta - \delta_1) < 1 + \tfrac{1}{\kappa^2}$ and $\beta_2 = q_1 \delta_1 < 1 + \tfrac{1}{\kappa^2}$.
Applying H\"older we obtain the bound
\begin{multline}
  \label{eq:verysmalleta2}
  \sum_{\substack{K_1 + K_2 \geq K-1 \\ K_2 \geq AK}} \left(\int_{1\leq |\xi| \leq L} \alpha_1(\hat\pi_1(g_0)n(\xi) a(t))^{\beta_1}\dd \xi\right)^{\frac{1}{2p_1}} \\
    \times \left(\int_{1\leq |\xi| \leq L} \alpha_1(\hat\pi_3(g_0)n(\xi) a(t))^{\beta_1} \chi_{K_2}(g_0n(\xi)a(t)) \dd \xi\right)^{\frac{1}{2p_1}} \\
  \times \left(\int_{1\leq |\xi| \leq L} \alpha_2(\pi_2(g_0)n(\xi) a(t))^{\beta_2} \chi_{K_1}(g_0n(\xi)a(t)) \dd \xi\right)^{1 - \frac{1}{p_1}},
\end{multline}
and applying Proposition \ref{prop:rank2avoidance}, we obtain
\begin{equation}
  \label{eq:verysmalleta2eta1sum}
  L \e^{O(s)}\sum_{\substack{K_1 + K_2 \geq K-1 \\ K_2 \geq AK}}  \e^{-\frac{1}{2\kappa^2p_1}K_2s} \e^{-\frac{1}{\kappa^2q_1}K_1s} \\
  \ll L \e^{-((\frac{1}{\kappa^2}(1 - \frac{1}{p_1}) + \frac{A}{\kappa^2}(\frac{3}{2 p_1} - 1))K + O(1))s}
\end{equation}
assuming $p_1 < \tfrac{3}{2}$. 
We note that if one takes $\kappa$ and $\beta$ arbitrarily close to $1$, then $q_1$ can be taken arbitrarily large.
Taking $A$ sufficiently large, one can choose $p$ sufficiently close to $2$ so that \eqref{eq:avoidanceMargulis} holds with $\mu > \tfrac{1}{3}$.

\section*{Data access statement}
No new data were generated or analysed during this study.

\bibliographystyle{amsalpha}
\bibliography{references}
\def\cprime{$'$} \def\cprime{$'$} \def\cprime{$'$}
\providecommand{\bysame}{\leavevmode\hbox to3em{\hrulefill}\thinspace}
\providecommand{\MR}{\relax\ifhmode\unskip\space\fi MR }
\providecommand{\MRhref}[2]{%
  \href{http://www.ams.org/mathscinet-getitem?mr=#1}{#2}
}

\end{document}